\documentclass[a4paper,12pt]{article}
\usepackage{latexsym}
\usepackage{amsmath}
\usepackage{amssymb}
\usepackage{enumerate}
\usepackage{bm}
\usepackage{mathrsfs}

\usepackage{theorem}
\newtheorem{theorem}{Theorem}[section]
\newtheorem{proposition}[theorem]{Proposition}
\newtheorem{lemma}[theorem]{Lemma}
\newtheorem{corollary}[theorem]{Corollary}

\theorembodyfont{\rmfamily}
\newtheorem{remark}[theorem]{Remark}

\newtheorem{definition}[theorem]{Definition}
\newtheorem{step}{Step}

\newtheorem{proof}{\textmd{\textit{Proof.}}}

\makeatletter

\@addtoreset{equation}{section}
\makeatother

\newcommand{\qedd}{\hfill \square}
\newcommand{\ve}{\varepsilon}
\newcommand{\ez}{\epsilon}
\newcommand{\del}{\partial}
\newcommand{\lra}{\longrightarrow}
\newcommand{\e}{\mathrm{e}}
\def\td{\mathrm{d}}

\newcommand{\N}{\ensuremath{\mathbb{N}}}
\newcommand{\Z}{\ensuremath{\mathbb{Z}}}
\newcommand{\R}{\ensuremath{\mathbb{R}}}

\newcommand{\cA}{\ensuremath{\mathcal{A}}}

\newcommand{\fm}{\ensuremath{\mathfrak{m}}}
\newcommand{\bb}{\ensuremath{\mathbf{b}}}

\newcommand{\sB}{\ensuremath{\mathsf{B}}}
\newcommand{\sI}{\ensuremath{\mathsf{I}}}
\newcommand{\sJ}{\ensuremath{\mathsf{J}}}
\newcommand{\sR}{\ensuremath{\mathsf{R}}}

\newcommand{\LL}{\ensuremath{\mathscr{L}}}

\def\vol{\mathop{\mathrm{vol}}\nolimits}

\def\Ric{\mathop{\mathrm{Ric}}\nolimits}
\def\trace{\mathop{\mathrm{trace}}\nolimits}

\newcommand{\Grad}{\bm{\nabla}}
\newcommand{\Lap}{\bm{\Delta}}

\newcommand{\wt}[1]{\widetilde{#1}}
\newcommand{\rev}[1]{\overleftarrow{#1}}
\newcommand{\ol}[1]{\overline{#1}}

\title{Geometry of weighted Lorentz--Finsler manifolds~II:\\ A splitting theorem}

\author{Yufeng LU\thanks{
Department of Mathematics, Osaka University, Osaka 560-0043, Japan
({\sf yufenglu.math@gmail.com}, {\sf s.ohta@math.sci.osaka-u.ac.jp})}
\and
Ettore MINGUZZI\thanks{
Dipartimento di Matematica e Informatica ``U. Dini'', Universit\`a degli Studi di Firenze,
Via S.~Marta 3, I-50139 Firenze, Italy ({\sf ettore.minguzzi@unifi.it})}
\and
Shin-ichi OHTA\footnotemark[1] \textsuperscript{,}\thanks{
RIKEN Center for Advanced Intelligence Project (AIP),
1-4-1 Nihonbashi, Tokyo 103-0027, Japan}}

\date{}
\begin{document}

\maketitle

\begin{abstract}
We show an analogue of the Lorentzian splitting theorem
for weighted Lorentz--Finsler manifolds:
If a weighted Berwald spacetime of nonnegative weighted Ricci curvature
satisfies certain completeness and metrizability conditions and includes a timelike straight line,
then it necessarily admits a one-dimensional family of isometric translations
generated by the gradient vector field of a Busemann function.
Moreover, our formulation in terms of the $\ez$-range introduced in our previous work enables us
to unify the previously known splitting theorems for weighted Lorentzian manifolds
by Case and Woolgar--Wylie into a single framework.
\end{abstract}

\section{Introduction}

The purpose of this work is to develop an analogue of the \emph{Lorentzian splitting theorem}
for weighted Lorentz--Finsler manifolds.
The Lorentzian version reads as follows.

\begin{theorem}[Lorentzian splitting theorem]\label{th:Lsplit}
Let $(M,g)$ be a connected, timelike geodesically complete spacetime which satisfies
the strong energy condition, i.e., $\Ric(v) \ge 0$ for all timelike vectors $v$.
If $(M,g)$ includes a timelike straight line,
then it is isometric to $(\mathbb{R} \times \Sigma, -\td t^2 + h)$,
where $(\Sigma, h)$ is a complete Riemannian manifold.
\end{theorem}

This splitting theorem was conjectured by Yau \cite{Ya} in 1982
and established by Newman \cite{Ne} in 1990.
We refer to \cite{GH} for a simplified proof (see also \cite[Chapter~14]{BEE}),
\cite{Ga-split} for a variant splitting theorem assuming the global hyperbolicity
instead of the geodesic completeness, and to \cite{Fl,Ga-rela} for surveys
including some similarities and differences between Theorem~\ref{th:Lsplit}
and the \emph{Cheeger--Gromoll splitting theorem} \cite{CG} in Riemannian geometry.
Some generalizations to \emph{weighted Lorentzian manifolds}
can be found in \cite{Ca,WW2}.

\emph{Lorentz--Finsler manifolds} generalize Lorentzian manifolds
in the same manner as Finsler manifolds generalize Riemannian manifolds
(see Section~\ref{sc:pre} for the precise definition).
This class of space(time)s has been attracting growing interest from
the viewpoint of the investigation of less regular spacetimes
as well as the synthetic geometric study of Lorentzian manifolds.
The latter study was motivated by the fruitful theory in the positive-definite case,
namely the \emph{curvature-dimension condition}
as a synthetic notion of lower (weighted) Ricci curvature bound.
Finsler manifolds of weighted Ricci curvature bounded below
provide an important ``non-Riemannian'' class of spaces satisfying
the curvature-dimension condition (see \cite{Oint,Obook}),
and this fact motivated the introduction of a reinforced condition called
the \emph{Riemannian curvature-dimension condition}.
We refer to a survey \cite{Am} for the successful theory of metric measure spaces
satisfying the Riemannian curvature-dimension condition.
In the recent development of a Lorentzian counterpart to the curvature-dimension condition
(see, e.g., \cite{CM,Mc}),
(weighted) Lorentz--Finsler manifolds will have a similar significance to Finsler manifolds,
and then it is important to know what results in Lorentzian geometry
can be generalized to the Lorentz--Finsler setting,
which is an aim of our work.

In our previous papers \cite{LMO,LMO2},
we have investigated time oriented Lorentz--Finsler manifolds
equipped with weight functions on causal vectors,
that we call \emph{weighted Finsler spacetimes} (see Section~\ref{sc:wLF} for the precise setting).
We established several singularity theorems in \cite{LMO}
and comparison theorems including the Laplacian comparison theorem in \cite{LMO2}
(see Theorem~\ref{th:Lcomp}).
On one hand, the Laplacian comparison theorem is a key ingredient
in the proof of the Lorentzian splitting theorem.
On the other hand, the Cheeger--Gromoll splitting theorem was generalized to
(weighted) Finsler manifolds in \cite{Osplit}
(with the help of the Laplacian comparison theorem in \cite{OShf}).
Hence it is natural to proceed to splitting theorems for weighted Finsler spacetimes.

Our main result is the following.

\begin{theorem}[Splitting theorem; Theorem~\ref{th:Gsplit}]\label{th:LFsplit}
Let $(M,L,\Psi)$ be a connected, timelike geodesically complete, weighted Berwald spacetime
of $\Ric_N \ge 0$ on $\Omega$ with $N \in (-\infty,0) \cup [n,+\infty]$,
where $\dim M=n+1$.
Suppose also that
\begin{enumerate}[$(1)$]
\item\label{cplt}
the $\ez$-completeness holds for some $\ez$ in the corresponding $\ez$-range \eqref{eq:erange},
\item\label{line}
there is a timelike straight line $\eta:\R \lra M$,
\item\label{met}
there is a timelike geodesically complete Lorentzian structure $g$ of $M$
whose Levi-Civita connection coincides with the Chern connection of $L$
such that $\eta$ is timelike for $g$,
\item\label{cover}
in the universal cover $(\wt{M},\tilde{g})$ of $(M,g)$,
a lift of $\eta$ is a timelike straight line.
\end{enumerate}
Then the Lorentzian manifold $(M,g_{\Grad \ol{\bb}{}^{\eta}})$ isometrically splits
in the sense that there exists an $n$-dimensional complete Riemannian manifold $(\Sigma,h)$
such that $(M,g_{\Grad \ol{\bb}{}^{\eta}})$ is isometric to $(\R \times \Sigma,-\td t^2 +h)$.

Furthermore, for every $x \in \Sigma$,
the image $\zeta(t)$ of $(t,x)$ is a straight line bi-asymptotic to $\eta$
and $\Psi$ is constant on $\zeta$.
\end{theorem}

Let us quickly explain the terms in the theorem.
A \emph{Berwald spacetime} is a special kind of Finsler spacetime
having various fine properties (see Definition~\ref{df:Ber}),
$\Psi:M \lra \R$ is a \emph{weight function} and $\Ric_N$ is
the associated \emph{weighted Ricci curvature} (Definition~\ref{df:wRic}),
and $\Omega$ is the set of future-directed timelike vectors.
The Lorentzian structure $g_{\Grad \ol{\bb}{}^{\eta}}$
($x \longmapsto g_{\Grad \ol{\bb}{}^{\eta}(x)}$) is the second order approximation
of the original Lorentz--Finsler structure $L$ in the direction of
the gradient vector field $\Grad \ol{\bb}{}^{\eta}$ of the Busemann function
$\ol{\bb}{}^{\eta}$ for the ray $\bar{\eta}$ in the past direction
(see \eqref{eq:g_v}, \eqref{eq:grad} and Subsection~\ref{ssc:lines}).
The existence of $g$ as in \eqref{met} can be regarded as
a Lorentzian analogue of Szab\'o's \emph{metrizability theorem}.
On one hand, it is not known if it holds true in our setting (see Remark~\ref{rm:met}).
On the other hand, the metrizability is a necessary condition for the splitting as above
(one can take $g=g_{\Grad \ol{\bb}{}^{\eta}}$,
while $g$ in Theorem~\ref{th:LFsplit} may not coincide with $g_{\Grad \ol{\bb}{}^{\eta}}$).
The last hypothesis \eqref{cover} is needed
merely for technical reasons and likely redundant.

In addition to the generalization from Lorentzian to Lorentz--Finsler manifolds,
our approach using the \emph{$\ez$-range} (introduced in \cite{LMO}; see \eqref{eq:erange})
is more general than \cite{Ca,WW2} for weighted Lorentzian manifolds.
The \emph{$\ez$-completeness} (Definition~\ref{df:cmplt}) in the hypothesis \eqref{cplt}
is an appropriate generalization
of the timelike geodesic completeness involving the weight function $\Psi$,
and the use of the $\ez$-range enables us to unify the splitting theorems
in \cite{Ca,WW2} into a single framework
($N \in [n,+\infty)$ with $\ez=1$ recovers \cite{Ca},
and $N \in (-\infty,0) \cup \{+\infty\}$ with $\ez=0$ corresponds to \cite{WW2}).
We remark that $\Ric_N$ in this paper corresponds to $\Ric_f^{N+1}$ in \cite{WW2}
(see Remark~\ref{rm:N}).

The main idea of the splitting theorem is extending the line $\eta$
into a congruence of lines (precisely, asymptotes of $\eta$).
Such lines in fact coincide with integral curves of $\Grad \ol{\bb}{}^{\eta}$.
A distinct feature of the Lorentz--Finsler version is that
the splitting is done for the Lorentzian structure $g_{\Grad \ol{\bb}{}^{\eta}}$
and the original Lorentz--Finsler structure $L$ itself does not possess a product structure.
This is actually a similar phenomenon to the Finsler case;
note that normed spaces do not isometrically split.
Nonetheless, we can show the following in the same spirit as \cite[\S 5]{Osplit}.

\begin{corollary}[Isometric translations; Corollary~\ref{cr:Gtrans}]\label{cr:LFtrans}
Let $(M,L,\Psi)$ be as in Theorem~$\ref{th:LFsplit}$.
Then, in the product structure $M=\R \times \Sigma$, the translations
\[ \Phi_t(s,x):=(s+t,x), \qquad (s,x) \in \R \times \Sigma,\ t \in \R, \]
are isometric transformations of $(M,L)$.
\end{corollary}

As a weight, one can also employ a function $\psi_{\fm}$ on causal vectors
induced from a measure $\fm$ on $M$ in the same manner as \cite{Oint} in the Finsler setting
(see Remarks~\ref{rm:meas}, \ref{rm:psi_m} and \ref{rm:msplit} for details).
In this case, $\Phi_t$ is measure-preserving (see Remark~\ref{rm:msplit}).
We remark that the Berwald condition is essential at least in Corollary~\ref{cr:LFtrans},
due to the extreme flexibility of general Lorentz--Finsler structures.
The timelike congruence extending $\eta$ probes the indicatrix
only near the vector field $\Grad \ol{\bb}{}^{\eta}$.
Thus, from the ensuing data of $g_{\Grad \ol{\bb}{}^{\eta}}$,
one cannot recover the indicatrix
far from $\Grad \ol{\bb}{}^{\eta}$ (see Remark~\ref{rm:nonBer} for details).

As for what concerns our strategy,
we first analyze rays and associated Busemann functions along the lines of \cite{GH}.
Then, to avoid the use of Bartnik's existence theorem \cite[Theorem 4.1]{Ba}
of maximal spacelike hypersurfaces which has no analogue in the Lorentz--Finsler setting at present,
we follow a more direct argument given in \cite{Es} to show the splitting theorem.
We remark that our argument will be indebted to various fine properties of Berwald spacetimes,
and it is unclear if Theorem~\ref{th:LFsplit} can be generalized to non-Berwald Finsler spacetimes
(cf.\ \cite[\S 4]{Osplit} in the Finsler setting).

This article is organized as follows.
We review necessary notions for Lorentz--Finsler manifolds in Section~\ref{sc:pre}
and discuss the weighted situation in Section~\ref{sc:wLF}.
We study rays and associated Busemann functions
in Sections~\ref{sc:ray} and \ref{sc:busemann}, respectively.
Then Section~\ref{sc:max} is devoted to a key step concerning
the behavior of the Busemann functions near the straight line.
We finally prove the splitting theorem in Section~\ref{sc:split}.

\section{Preliminaries for Finsler spacetimes}\label{sc:pre}

We briefly recall necessary notions for Lorentz--Finsler manifolds
(see also \cite{LMO,LMO2}).
One can in fact discuss similarly to the positive-definite case
(see, e.g., \cite{BCS,Obook,Shlec}) to some extent.
We refer to \cite{BEE,ON} for the basics of Lorentzian geometry
and to \cite{Min-spray,Min-causality,Min-Rev} for some generalizations
including Lorentz--Finsler manifolds.

Throughout the article,
let $M$ be a connected $C^{\infty}$-manifold without boundary of dimension $n+1$ ($\ge 2$).
Given a local coordinate system $(x^\alpha)_{\alpha=0}^n$ on an open set $U\subset M$,
we will use the fiber-wise linear coordinates
$v =\sum_{\alpha=0}^n v^{\alpha} (\partial/\partial x^{\alpha})|_x,$ $x \in U$.
%
%
We follow Beem's definition \cite{Be} of a Finsler version of Lorentzian manifolds as follows.

\begin{definition}[Lorentz--Finsler structures]\label{df:LFstr}
A \emph{Lorentz--Finsler structure} of $M$ is a function
$L:TM \lra \R$ satisfying the following conditions:
\begin{enumerate}[(1)]
\item $L \in C^{\infty}(TM \setminus 0)$, where $0$ denotes the zero section;
\item $L(cv)=c^2 L(v)$ for all $v \in TM$ and $c>0$;
\item For any $v \in TM \setminus 0$, the symmetric matrix
\begin{equation}\label{eq:g_ij}
\big( g_{\alpha \beta}(v) \big)_{\alpha,\beta=0}^n
 :=\bigg( \frac{\del^2 L}{\del v^\alpha \del v^\beta}(v) \bigg)_{\alpha,\beta=0}^n
\end{equation}
is non-degenerate with signature $(-,+,\ldots,+)$.
\end{enumerate}
Then we call $(M,L)$ a ($C^{\infty}$-)\emph{Lorentz--Finsler manifold}.
\end{definition}

We say that $(M,L)$ is \emph{reversible} if $L(-v)=L(v)$ for all $v\in TM$.
For $v \in T_xM \setminus \{0\}$, we define a Lorentzian metric $g_v$ of $T_xM$
by using \eqref{eq:g_ij} as
\begin{equation}\label{eq:g_v}
g_v \Bigg( \sum_{\alpha=0}^n a_\alpha \frac{\del}{\del x^\alpha}\bigg|_x,
 \sum_{\beta=0}^n b_\beta \frac{\del}{\del x^\beta}\bigg|_x \Bigg)
 :=\sum_{\alpha,\beta=0}^n g_{\alpha \beta}(v) a_\alpha b_\beta.
\end{equation}
Then we have $g_v(v,v)=2L(v)$ by Euler's homogeneous function theorem.

A tangent vector $v \in TM $ is said to be \emph{timelike} (resp.\ \emph{null})
if $L(v)<0$ (resp.\ $L(v)=0$).
We say that $v$ is \emph{lightlike} if it is null and nonzero,
and \emph{causal} (or \emph{non-spacelike}) if it is timelike or lightlike
(i.e., $L(v) \le 0$ and $v \neq 0$).
\emph{Spacelike} vectors are those for which $L(v)>0$ or $v=0$.
We denote by $\Omega'_x \subset T_xM$ the set of timelike vectors and
$\Omega' :=\bigcup_{x \in M} \Omega'_x$.
For later use, we define, for causal vectors $v$,
\begin{equation}\label{eq:LtoF}
F(v) :=\sqrt{-2L(v)} =\sqrt{-g_v(v,v)}.
\end{equation}

Note that $\Omega'_x \neq \emptyset$, every connected component of $\Omega'_x$
is a convex cone (see \cite{Be}, \cite[Lemma~2.3]{LMO}),
and that the closures of different components intersect only at $0$
(see \cite[Proposition~1]{Min-cone}).
In general, the number of connected components of $\Omega_x'$ may be larger than 2
(see \cite{Be}, \cite[Example~2.4]{LMO}).
This fact will not affect our discussion because we shall deal with
only future-directed (timelike or causal) vectors.
We also remark that $\Omega'_x$ has exactly two connected components
in reversible Lorentz--Finsler manifolds of dimension $\ge 3$
(see \cite[Theorem~7]{Min-cone}).

\begin{definition}[Finsler spacetimes]\label{df:spacetime}
If a Lorentz--Finsler manifold $(M,L)$ admits a smooth timelike vector field $X$,
then $(M,L)$ is said to be \emph{time oriented} (by $X$).
A time oriented Lorentz--Finsler manifold will be called a \emph{Finsler spacetime}.
\end{definition}

In a Finsler spacetime oriented by $X$,
a causal vector $v \in T_xM$ is said to be \emph{future-directed}
if it lies in the same connected component of $\overline{\Omega'}\!_x \setminus \{0\}$ as $X(x)$.
We denote by $\Omega_x \subset \Omega'_x$ the set of future-directed timelike vectors,
and define
\[ \Omega :=\bigcup_{x \in M} \Omega_x, \qquad
 \overline{\Omega} :=\bigcup_{x \in M} \overline{\Omega}_x, \qquad
 \overline{\Omega} \setminus 0 :=\bigcup_{x \in M} (\overline{\Omega}_x \setminus \{0\}). \]
A $C^1$-curve in $(M,L)$ is said to be \emph{timelike} (resp.\ \emph{causal})
if its tangent vector is always timelike (resp.\ causal).
All causal curves will be future-directed.

Given $x,y \in M$, we write $x \ll y$ (resp.\ $x<y$)
if there is a future-directed timelike (resp.\ causal) curve from $x$ to $y$,
and $x \le y$ means that $x=y$ or $x<y$.
Then we define the \emph{chronological past} and \emph{future} of $x$ by
\[ I^-(x):=\{y \in M \,|\, y \ll x\}, \qquad I^+(x):=\{y \in M \,|\, x \ll y\}. \]

We also define the \emph{Lorentz--Finsler distance} $d(x,y)$ for $x,y \in M$ by
\[ d(x,y) :=\sup_{\eta} \ell(\eta), \qquad
 \ell(\eta) :=\int_0^1 F\big( \dot{\eta}(t) \big) \,\td t, \]
where $\eta:[0,1] \lra M$ runs over all causal curves from $x$ to $y$
(recall \eqref{eq:LtoF} for $F$).
We set $d(x,y):=0$ if there is no causal curve from $x$ to $y$ (namely $x \not< y$).
A constant speed causal curve attaining the above supremum,
which is a causal \emph{geodesic}, is said to be \emph{maximizing}.
In general, causal geodesics are locally maximizing in the same way
as geodesics are locally minimizing in Riemannian or Finsler geometry.

In general, the distance function $d$ is only lower semi-continuous and can be infinite
(see, e.g., \cite[Proposition~6.7]{Min-Ray}).
In \emph{globally hyperbolic} Finsler spacetimes (see \cite{LMO} for the definition),
$d$ is finite and continuous,
and any pair of points $x,y \in M$ with $x<y$ admits a maximizing geodesic from $x$ to $y$
(see \cite[Propositions~6.8, 6.9]{Min-Ray}).
In what follows, however, we will not assume the global hyperbolicity.


Next we introduce the covariant derivative and the Ricci curvature.
Define
\[ \gamma^{\alpha}_{\beta \delta} (v)
 :=\frac{1}{2} \sum_{\lambda=0}^n g^{\alpha \lambda}(v)
 \bigg\{ \frac{\del g_{\lambda \delta}}{\del x^{\beta}}(v) +\frac{\del g_{\beta \lambda}}{\del x^{\delta}}(v)
 -\frac{\del g_{\beta \delta}}{\del x^{\lambda}}(v) \bigg\} \]
for $\alpha,\beta,\delta =0,1,\ldots,n$ and $v \in TM\setminus 0$,
where $(g^{\alpha \beta}(v))$ is the inverse matrix of $(g_{\alpha\beta}(v))$,
\[ G^{\alpha}(v) :=\frac{1}{2}\sum_{\beta,\delta=0}^n \gamma^{\alpha}_{\beta \delta}(v) v^{\beta} v^{\delta},
 \qquad N^{\alpha}_{\beta}(v) :=\frac{\del G^{\alpha}}{\del v^{\beta}}(v) \]
for $v \in TM \setminus 0$ ($G^{\alpha}(0)=N^{\alpha}_{\beta}(0):=0$ by convention), and
\begin{equation}\label{eq:Gamma}
\Gamma^{\alpha}_{\beta \delta}(v):=\gamma^{\alpha}_{\beta \delta}(v)
 -\frac{1}{2}\sum_{\lambda,\mu=0}^n g^{\alpha \lambda}(v)
 \bigg( \frac{\del g_{\lambda \delta}}{\del v^{\mu}}N^{\mu}_{\beta}
 +\frac{\del g_{\beta \lambda}}{\del v^{\mu}} N^{\mu}_{\delta}
 -\frac{\del g_{\beta \delta}}{\del v^{\mu}} N^{\mu}_{\lambda} \bigg)(v)
\end{equation}
on $TM \setminus 0$.
Then the \emph{covariant derivative} of a vector field
$X=\sum_{\alpha=0}^n X^{\alpha} (\del/\del x^{\alpha})$ is defined as
\[ D_v^w X :=\sum_{\alpha,\beta=0}^n
 \Bigg\{ v^{\beta} \frac{\del X^{\alpha}}{\del x^{\beta}}(x)
 +\sum_{\delta=0}^n \Gamma^{\alpha}_{\beta \delta}(w) v^{\beta} X^{\delta}(x) \Bigg\}
 \frac{\del}{\del x^{\alpha}} \bigg|_x \]
for $v \in T_xM$ with reference vector $w \in T_xM \setminus \{0\}$.
We remark that the functions $\Gamma^{\alpha}_{\beta \delta}$ in \eqref{eq:Gamma}
are the coefficients of the \emph{Chern}(--\emph{Rund}) \emph{connection}.

In the Lorentzian case, $g_{\alpha\beta}$ is constant in each tangent space
(thus $\Gamma^{\alpha}_{\beta\delta}=\gamma^{\alpha}_{\beta\delta}$)
and the covariant derivative does not depend on the choice of a reference vector.
In the Lorentz--Finlser setting, the following class is worth considering.

\begin{definition}[Berwald spacetimes]\label{df:Ber}
A Finsler spacetime $(M,L)$ is said to be of \emph{Berwald type}
(or called a \emph{Berwald spacetime}) if $\Gamma^{\alpha}_{\beta \delta}$ is constant
on the slit tangent space $T_xM \setminus \{0\}$ for any $x$
in the domain of every local coordinate system.
\end{definition}

The Berwald condition is strong but provides a reasonable (nontrivial) class
where we can mimic Lorentzian techniques.
By the very definition, the covariant derivative on a Berwald spacetime
is defined independently from the choice of a reference vector.
Examples of Berwald spacetimes include Lorentzian manifolds
and flat Lorentz--Finsler structures of $\R^{n+1}$
(every tangent space $T_x\R^{n+1}$ is canonically isometric to $T_0\R^{n+1}$
and we have $\Gamma^{\alpha}_{\beta\delta}=\gamma^{\alpha}_{\beta\delta}=0$).
We refer to \cite{FHPV,FP,FPP,GT,HPV} for some investigations on Berwald spacetimes,
and to \cite[Chapter~10]{BCS} for the positive-definite case.

\begin{remark}[Metrizability]\label{rm:met}
In the positive-definite case,
Szab\'o showed that a Berwald space $(M,F)$ admits a Riemannian metric $g$
whose Levi-Civita connection coincides with the Chern connection,
i.e., the Christoffel symbols of $g$ coincide with $\Gamma^i_{jk}$ of $F$
(see \cite{Sz} and \cite[Exercise~10.1.4]{BCS}).
This is called the \emph{metrizability theorem}.
It is not known whether the metrizability can be generalized to the Lorentz--Finsler setting.
In \cite{FHPV}, some counter-examples were constructed for Lorentz--Finsler structures
defined only on a subset of $TM$.
Their discussion is not applicable to Lorentz--Finsler structures defined on whole $TM$
as in Definition~\ref{df:LFstr}.
\end{remark}

The \emph{geodesic equation} for a causal curve $\eta:[0,1] \lra M$
is written as $D^{\dot{\eta}}_{\dot{\eta}}\dot{\eta} \equiv 0$.
The \emph{exponential map} is defined in the same way as the Riemannian case,
which is $C^{\infty}$ on a neighborhood of the zero section only in Berwald spacetimes
(see, e.g., \cite{Min-coord} as well as
\cite[Exercise~5.3.5]{BCS} in the positive-definite case).

Next we introduce the Ricci curvature.
A $C^{\infty}$-vector field $J$ along a geodesic $\eta$ is called a \emph{Jacobi field}
if it satisfies $D^{\dot{\eta}}_{\dot{\eta}} D^{\dot{\eta}}_{\dot{\eta}} J +R_{\dot{\eta}}(J) =0$,
where
\[ R_v(w):=\sum_{\alpha,\beta=0}^n R^{\alpha}_{\beta}(v) w^{\beta} \frac{\del}{\del x^{\alpha}} \]
for $v,w \in T_xM$ and
\[ R^{\alpha}_{\beta}(v) :=2\frac{\del G^{\alpha}}{\del x^{\beta}}(v)
 -\sum_{\delta=0}^n \bigg\{ \frac{\del N^{\alpha}_{\beta}}{\del x^{\delta}}(v) v^{\delta}
 -2\frac{\del N^{\alpha}_{\beta}}{\del v^{\delta}}(v) G^{\delta}(v) \bigg\}
 -\sum_{\delta=0}^n N^{\alpha}_{\delta}(v) N^{\delta}_{\beta}(v) \]
is the \emph{curvature tensor}.
A Jacobi field is also characterized as the variational vector field of a geodesic variation.

Note that $R_v(w)$ is linear in $w$, thereby $R_v:T_xM \lra T_xM$ is an endomorphism
for each $v \in T_xM$.
For $v \in \overline{\Omega}_x$, we define the \emph{Ricci curvature}
(or \emph{Ricci scalar}) of $v$ as the trace of $R_v$: $\Ric(v):=\trace(R_v)$.
We remark that $\Ric(cv)=c^2 \Ric(v)$ for $c>0$.


In order to introduce the spacetime Laplacian,
we consider the dual structure to $L$ and the Legendre transform
(see \cite{Min-cone}, \cite[\S 3.1]{Min-causality}, \cite[\S 4.4]{LMO2} for further discussions).
Define the \emph{polar cone} to $\Omega_x$ by
\[ \Omega^*_x :=\big\{ \omega \in T_x^*M \,|\,
 \omega(v)<0\ \text{for all}\ v \in \overline{\Omega}_x \setminus \{0\} \big\}. \]
This is an open convex cone in $T_x^*M$.
For $\omega \in \Omega^*_x$, we define
\[ L^*(\omega) := -\frac{1}{2}\bigg( \sup_{v \in \Omega_x \cap F^{-1}(1)} \omega(v) \bigg)^2
 =-\frac{1}{2}\inf_{v \in \Omega_x \cap F^{-1}(1)} \big( \omega(v) \big)^2. \]
Then we have the \emph{reverse Cauchy--Schwarz inequality}
$4L^*(\omega) L(v) \le (\omega(v))^2$
for $v \in \Omega_x$ and $\omega \in \Omega^*_x$,
and arrive at the following variational definition of the Legendre transform.

\begin{definition}[Legendre transform]\label{df:Leg}
Define the \emph{Legendre transform}
$\LL^*:\Omega^*_x \lra \Omega_x$ as the map sending $\omega \in \Omega^*_x$
to the unique element $v \in \Omega_x$ satisfying $L(v)=L^*(\omega)=\omega(v)/2$.
We also define $\LL^*(0):=0$.
\end{definition}

A coordinate expression of the Legendre transform is given by
\begin{equation}\label{eq:Leg}
\LL^*(\omega)
 =\sum_{\alpha=0}^n \frac{\del L^*}{\del \omega_{\alpha}}(\omega) \frac{\del}{\del x^{\alpha}},
 \qquad \text{where}\,\ \omega =\sum_{\alpha=0}^n \omega_{\alpha} \,\td x^{\alpha}.
\end{equation}
We refer to \cite{LMO2,Min-cone,Min-causality}
for some basic properties of the Legendre transforms.

A continuous function $f:M \lra \R$ is called a \emph{time function}
if $f(x)<f(y)$ for all $x,y \in M$ with $x<y$.
A $C^1$-function $f:M \lra \R$ is said to be \emph{temporal} if $-\td f(x) \in \Omega^*_x$
for all $x \in M$.
Observe that temporal functions are time functions.

For a temporal function $f:M \lra \R$,
define the \emph{gradient vector} of $-f$ at $x \in M$ by
\begin{equation}\label{eq:grad}
\Grad(-f)(x) :=\LL^* \big( {-\td f}(x) \big) \in \Omega_x.
\end{equation}
Note that, thanks to \eqref{eq:Leg}, we have
$g_{\Grad(-f)} (\Grad(-f)(x),v) =-\td f(v)$ for any $v \in T_xM$.
For a $C^2$-temporal function $f:M \lra \R$ and $x \in M$,
we define the \emph{Hessian} $\Grad^2 (-f):T_xM \lra T_xM$ by
\begin{equation}\label{eq:Hess}
\Grad^2(-f)(v) :=D^{\Grad(-f)}_v \big( \Grad(-f) \big).
\end{equation}
This spacetime Hessian has the following symmetry (see \cite[Lemma~4.12]{LMO2}):
\begin{equation}\label{eq:symm}
g_{\Grad(-f)}\big( \Grad^2 (-f)(v),w \big) =g_{\Grad(-f)}\big( v,\Grad^2(-f)(w) \big)
\end{equation}
for all $v,w \in T_xM$.
Then we define the spacetime \emph{Laplacian} (or \emph{d'Alembertian}) as the trace of the Hessian:
\begin{equation}\label{eq:Lap}
\Lap (-f):=\trace\big( \Grad^2(-f) \big).
\end{equation}
We remark that this Laplacian is not elliptic but hyperbolic,
and is nonlinear (since the Legendre transform is nonlinear).

\begin{remark}[Berwald case]\label{rm:Hess}
For $v \in T_xM$ and the geodesic $\eta:(-\ve,\ve) \lra M$ with $\dot{\eta}(0)=v$,
the second order derivative $(-f \circ \eta)''(0)$ does not coincide with
$g_{\Grad(-f)}(\Grad^2(-f)(v),v)$ in general.
They coincide in Berwald spacetimes thanks to the fiber-wise constancy of
the connection coefficients $\Gamma^{\alpha}_{\beta \delta}$
(see \cite[\S 12.1]{Obook} for the positive-definite case).
\end{remark}

\section{Weighted Finsler spacetimes}\label{sc:wLF}

Next, as a \emph{weight} on a Finsler spacetime $(M,L)$,
we employ a $C^{\infty}$-function $\Psi:M \lra \R$.
The following definition generalizes the weighted Ricci curvature for Finsler manifolds
introduced in \cite{Oint} (see also \cite{Oneg} for the case of $N<0$).

\begin{definition}[Weighted Ricci curvature]\label{df:wRic}
Given $v \in \overline{\Omega} \setminus 0$,
let $\eta:(-\ve,\ve) \lra M$ be the causal geodesic with $\dot{\eta}(0)=v$.
Then we define the \emph{weighted Ricci curvature} by
\[ \Ric_N(v) :=\Ric(v) +(\Psi \circ \eta)''(0) -\frac{(\Psi \circ\eta)'(0)^2}{N-n} \]
for $N \in \R \setminus \{n\}$.
We also define $\Ric_{\infty}(v) :=\Ric(v) +(\Psi \circ \eta)''(0)$,
$\Ric_n(v) :=\lim_{N \downarrow n} \Ric_N(v)$, and $\Ric_N(0):=0$.
\end{definition}

\begin{remark}\label{rm:N}
Since $\dim M=n+1$,
$\Ric_N$ in this paper corresponds to $\Ric_f^{N+1}$ in \cite{WW2}.
\end{remark}

\begin{remark}[Weight functions associated with measures]\label{rm:meas}
Instead of a function $\Psi$ on $M$ as above, one can alternatively employ a function
$\psi_{\fm}:\ol{\Omega} \setminus 0 \lra \R$ on causal vectors associated with a measure $\fm$ on $M$.
To be precise, $\psi_{\fm}$ is defined by
\[ \fm(\td x) =\e^{-\psi_{\fm} \circ \dot{\eta}} \sqrt{-\det\big[ \big( g_{\alpha\beta}(\dot{\eta}) \big) \big]}
 \,\td x^0 \td x^1 \cdots \td x^n \]
along causal geodesics $\eta$.
One way to unify these two cases is to consider a general $0$-homogeneous function
$\psi:\ol{\Omega} \setminus 0 \lra \R$ as in \cite{LMO,LMO2}
(we remark that, in general, such a function may not be associated with any measure;
see \cite{ORand,Obook} for the Finsler case).
However, since a weighted Laplacian corresponding to general $\psi$ is yet to be developed
(even for Riemannian manifolds), we restrict ourselves to $\Psi$ or $\psi_{\fm}$.
\end{remark}

We will say that $\Ric_N \ge K$ holds \emph{in timelike directions} for some $K \in \R$
if we have $\Ric_N(v) \ge KF^2(v) =-2KL(v)$ for all $v \in \Omega$ (recall \eqref{eq:LtoF}).
Observe that
\[ \Ric_n(v) \le \Ric_N(v) \le \Ric_{\infty}(v) \le \Ric_{N'}(v) \]
holds for $n<N<\infty$ and $-\infty<N'<n$.
Associated with the real parameter $N$, the following notion was introduced in \cite{LMO}.

\begin{definition}[$\ez$-range]\label{df:eran}
Given $N \in (-\infty,0] \cup [n,+\infty]$, we will consider $\ez \in \R$ in the following
\emph{$\ez$-range}:
\begin{equation}\label{eq:erange}
\epsilon=0 \,\text{ for } N=0, \qquad
 |\epsilon| < \sqrt{\frac{N}{N-n}} \,\text{ for } N \neq 0,n, \qquad
 \ez \in \R \,\text{ for } N=n.
\end{equation}
The associated constant $c =c(N,\ez)$ is defined as
\begin{equation}\label{eq:LF-c}
c(N,\ez):= \frac{1}{n}\bigg( 1-\ez^2\frac{N-n}{N} \bigg) >0 \,\text{ for } N \neq 0, \qquad
 c(0,0):= \frac{1}{n}.
\end{equation}
\end{definition}

Note that $\ez=1$ is admissible only for $N \in [n,+\infty)$ and $\ez=0$ is always admissible.
By the $\ez$-range, we can unify results for constant and variable curvature bounds
into a single framework.
See \cite{LMO} for singularity theorems and \cite{LMO2} for comparison theorems.

A Laplacian comparison theorem with $\ez$-range was established in \cite{LMO2} as follows.
Given $z \in M$ and any maximizing timelike geodesic $\eta:[0,T) \lra M$ with $\eta(0)=z$,
the function $u(x):=d(z,x)$ satisfies $-\td u(x) \in \Omega_x^*$ for $x \in \eta((0,T))$
and hence $\Lap(-u)(x)$ as in \eqref{eq:Lap} is well-defined.
We define the \emph{$\Psi$-Laplacian} (or the \emph{weighted Laplacian}) of $-u$ by
\[ \Lap\!^{\Psi}(-u)(x) :=\Lap(-u)(x) -\td\Psi \big( \Grad(-u)(x) \big). \]
Then the \emph{Laplacian comparison theorem} (\cite[Theorem~5.8]{LMO2})
for nonnegatively curved spacetimes asserts the following.

\begin{theorem}[Laplacian comparison theorem]\label{th:Lcomp}
Let $(M,L,\Psi)$ be a weighted Finsler spacetime of dimension $n+1 \ge 2$,
$N \in (-\infty,0] \cup [n,+\infty]$, and $\ez \in \R$ belong to the $\ez$-range \eqref{eq:erange}.
Suppose that $\Ric_N \ge 0$ holds in timelike directions.
Then, for any $z \in M$, the distance function $u(x):=d(z,x)$ satisfies
\begin{equation}\label{eq:Lcomp}
\Lap\!^{\Psi}(-u)\big( \eta(t) \big)
 \le \e^{\frac{2(\ez -1)}{n} \Psi(\eta(t))}
 \bigg( c\int_0^t \e^{\frac{2(\ez -1)}{n} \Psi(\eta(s))} \,\td s \bigg)^{-1}
\end{equation}
for any maximizing timelike geodesic $\eta:[0,T) \lra M$ of unit speed
$($i.e., $F(\dot{\eta})\equiv 1)$ emanating from $z$
and any $t \in (0,T)$, where $c$ is as in \eqref{eq:LF-c}.
\end{theorem}

Precisely, we considered a general $0$-homogeneous function
$\psi:\ol{\Omega} \setminus 0 \lra \R$ in \cite{LMO2} (recall Remark~\ref{rm:meas}).
Note also that \eqref{eq:Lcomp} is actually an intermediate estimate
in the proof of \cite[Theorem~5.8]{LMO2}.
We refer to \cite{Ca,WW2} for the weighted Lorentzian case.

We also explain the reverse version of Theorem~\ref{th:Lcomp} for thoroughness.
The \emph{reverse Lorentz--Finsler structure} of $L$ is defined as $\rev{L}(v):=L(-v)$,
and we put an arrow $\leftarrow$ on a quantity associated with $\rev{L}$.
The Lorentz--Finsler manifold $(M,\rev{L})$ is time oriented by $-X$,
so that $\rev{\Omega}=-\Omega$.
The corresponding weighted Laplacian is given as $\rev{\Lap}^{\Psi}f=-\Lap\!^{\Psi}(-f)$
for temporal functions $f$
(as for $\psi_{\fm}$, we have $\rev{\psi}_{\fm}(v)=\psi_{\fm}(-v)$).
We remark that, for the weighted Ricci curvature $\rev{\Ric}$ of $(M,\rev{L},\Psi)$,
we have $\rev{\Ric}_N(v)=\Ric_N(-v)$ for $v \in \rev{\Omega}$, and hence
the timelike curvature bound $\Ric_N \ge 0$ is equivalent to $\rev{\Ric}_N \ge 0$.

\begin{corollary}[Reverse version]\label{cr:Lcomp}
Let $(M,L,\Psi)$, $N$ and $\epsilon$ be as in Theorem~$\ref{th:Lcomp}$.
Then, for any $z \in M$, the function $\bar{u}(x):=d(x,z)$ satisfies
\begin{equation}\label{eq:Lcomp'}
\rev{\Lap}^{\Psi}(-\bar{u}) \big( \bar{\eta}(t) \big)
 =-\Lap\!^{\Psi} \bar{u} \big( \bar{\eta}(t) \big)
 \le \e^{\frac{2(\ez -1)}{n} \Psi({\bar{\eta}}(t))}
 \bigg( c\int_0^t \e^{\frac{2(\ez -1)}{n} \Psi(\bar{\eta}(s))} \,\td s \bigg)^{-1}
\end{equation}
for any maximizing timelike geodesic $\eta:(-T,0] \lra M$ of unit speed with $\eta(0)=z$
and any $t \in (0,T)$, where $\bar{\eta}(s):=\eta(-s)$.
\end{corollary}

Note that $\eta$ is a maximizing timelike geodesic of unit speed with respect to $L$
and so is $\bar{\eta}$ with respect to $\rev{L}$.
In the application of the Laplacian comparison theorem,
we need to require that the right-hand side of \eqref{eq:Lcomp} (and \eqref{eq:Lcomp'})
diverges to infinity.
To be precise, we will assume the following completeness condition introduced in \cite{LMO}.

\begin{definition}[$\ez$-completeness]\label{df:cmplt}
A weighted Finsler spacetime $(M,L,\Psi)$ is said to be
\emph{future timelike $\ez$-complete} if,
for any future inextendible timelike geodesic $\eta:[0,T) \lra M$,
\[ \lim_{t \to T} \int_0^t \e^{\frac{2(\ez -1)}{n} \Psi(\eta(s))} \,\td s =\infty \]
holds.
We say that $(M,L,\Psi)$ is \emph{timelike $\ez$-complete}
if both $(M,L,\Psi)$ and $(M,\rev{L},\Psi)$ are future timelike $\ez$-complete.
\end{definition}

The $1$-completeness corresponds to the usual geodesic completeness
(see Definition~\ref{df:gcmplt} below)
and the $0$-completeness recovers the $\Psi$-completeness introduced in \cite{Wy}
in the Riemannian case (see \cite{WW2} for the Lorentzian case).
A typical situation is that $\ez<1$, $\Psi$ is bounded above and
$(M,L)$ is (future) timelike geodesically complete.

\section{Rays and generalized co-rays}\label{sc:ray}

In this and the next sections, we shall follow the argument in \cite{GH}
to study rays and associated Busemann functions (see also \cite[Chapter~14]{BEE}).
We first recall the geodesic completeness conditions.
See \cite{Ga-cplt} for some connections between the global hyperbolicity
and the geodesic completeness.

\begin{definition}[Geodesic completeness]\label{df:gcmplt}
A Finsler spacetime $(M,L)$ is said to be \emph{future timelike geodesically complete}
if any timelike geodesic $\eta:[0,1] \lra M$ can be extended to a geodesic
$\tilde{\eta}:[0,\infty) \lra M$.
We say that $(M,L)$ is \emph{timelike geodesically complete}
if both $(M,L)$ and $(M,\rev{L})$ are future timelike geodesically complete
(in other words, the above $\eta$ is extended to a geodesic $\tilde{\eta}:\R \lra M$).
\end{definition}

A future inextendible causal geodesic $\eta:[0,\infty) \lra M$ is called a \emph{ray}
if each of its segments is maximizing, i.e.,
$\ell(\eta|_{[a,b]})=d(\eta(a),\eta(b))$ for all $0\le a\le b$.
%
For later convenience, we define
\begin{equation}\label{eq:I(eta)}
I^-(\eta) := \bigcup_{t>0} I^-\big( \eta(t) \big), \qquad
I(\eta) :=I^+\big( \eta(0) \big) \cap I^-(\eta).
\end{equation}
Note that $I(\eta)$ is an open set including $\eta((0,\infty))$.
Observe also that, for $x,y \in I(\eta)$ with $x \le y$,
we have $d(x,y)<\infty$ since the reverse triangle inequality yields
\begin{equation}\label{eq:GH2-4}
d\big( \eta(0),x \big) +d(x,y) +d\big( y,\eta(t) \big) \le d\big( \eta(0),\eta(t) \big)
\end{equation}
for sufficiently large $t$.

Let $(M,L)$ be future timelike geodesically complete in the remainder of the section.
Recall that, without the global hyperbolicity, connecting maximizing curves may not exist.
Thus the following notion will play a role.

\begin{definition}[Limit maximizing sequences]\label{df:limmax}
A sequence $\eta_k:[a_k,b_k] \lra M$ of causal curves
is said to be \emph{limit maximizing} if
$\ell(\eta_k) \ge d(\eta_k(a_k), \eta_k(b_k))-\delta_k$ holds for some $\delta_k \to 0$.
\end{definition}

Observe that, if $x<y$ and $d(x,y)<\infty$, then there exists a limit maximizing
sequence of causal curves from $x$ to $y$ by the definition of $d$.

We make use of the \emph{limit curve theorems} of Lorentzian geometry
in both the ``one endpoint'' and ``two endpoints'' cases
(see \cite{BEE,EG,Ga-curv,GH} and more recent \cite{Min-lct,Min-Rev}).
As observed in \cite{Min-Ray}, these limit curve theorems generalize to the Lorentz--Finsler setting
thanks to the existence of convex neighborhoods (see also \cite{Min-causality}
for the Finslerian statements of these results under weak differentiability conditions).

What is relevant for us is the fact that, due to the lower semi-continuity of $d$
and the upper semi-continuity of the length $\ell$
(with respect to the uniform convergence on compact intervals),
any limit maximizing sequence of causal curves admits a subsequence converging,
in a suitable parametrization, to a maximizing geodesic.
For instance, we have the following analogue of \cite[Proposition~2.3]{GH}.

\begin{lemma}\label{lm:GH2.3}
Let $\eta_k:[a_k,b_k] \lra M$ be a limit maximizing sequence of causal curves
converging to $\eta:[a,b] \lra M$ uniformly on some interval $[a,b] \subset \bigcap_k [a_k,b_k]$.
Then we have $\ell(\eta)=d(\eta(a),\eta(b))$, and hence $\eta$ is a maximizing curve
from $\eta(a)$ to $\eta(b)$.
\end{lemma}


The next lemma is an application of the limit curve theorem
(see also \cite[Lemma~2.4]{GH}).

\begin{lemma}\label{lm:GH2.4}
Let a sequence $(x_k)_{k \in \N}$ in $M$ converge to $x$
and take $y_k \in I^+(x_k)$ such that $d(x_k,y_k)<\infty$.
Let $\eta_k:[0,a_k] \lra M$ be a limit maximizing sequence of
causal curves with $\eta_k(0)=x_k$ and $\eta_k(a_k)=y_k$.
If $d(x_k,y_k) \to \infty$, then there exists a subsequence of $(\eta_k)_{k \in \N}$
with suitable reparametrizations converging uniformly on compact intervals
to a ray $\eta:[0,\infty) \lra M$ emanating from $x$.
\end{lemma}

We remark that the above ray $\eta$ may not be timelike.
This fact is at the origin of several complications in the proof of the splitting theorem.

Rays asymptotic to a given ray, called co-rays, play an important role in splitting theorems.
In the non-globally hyperbolic case, due to the possible absence of maximizing curves,
we need to consider also generalized co-rays.
They are defined as follows
(co-rays and generalized co-rays were introduced in \cite{BEMG} and \cite{EG}, respectively).

\begin{definition}[Generalized co-rays]\label{df:coray}
Let $\eta:[0,\infty) \lra M$ be a timelike ray and take $x \in I(\eta)$.
If a limit maximizing sequence of causal curves $\zeta_k:[0,a_k] \lra M$ satisfies
$\zeta_k(0) \to x$ and $\zeta_k(a_k)=\eta(t_k)$ for some $t_k>0$ with $t_k \to \infty$,
then its limit curve $\zeta:[0,\infty) \lra M$ is called a \emph{generalized co-ray} of $\eta$.
If each $\zeta_k$ is maximizing, then we call $\zeta$ a \emph{co-ray} of $\eta$.
A co-ray $\zeta$ with $\zeta_k(0)=x$ for all $k$ is called an \emph{asymptote} of $\eta$.
\end{definition}

Lemma~\ref{lm:GH2.4} ensures that the limit curve $\zeta$ is indeed a ray.
As another instance of the limit curve theorem,
we have the following analogue of \cite[Proposition~3.1]{GH}.

\begin{proposition}\label{pr:GH3.1}
Let $x,y \in M$ satisfy $y \in I^+(x)$ and $d(x,y)<\infty$,
and let $\eta_k:[0,a_k] \lra M$ be a limit maximizing sequence of causal curves
from $x$ to $y$.
Then $(\eta_k)_{k \in \N}$ admits a suitably reparametrized subsequence
converging uniformly on compact intervals
to either a lightlike ray or a maximizing curve from $x$ to $y$.
\end{proposition}

To avoid the convergence to a lightlike ray, we introduce the following condition.

\begin{definition}[Generalized timelike co-ray condition (GTCC)]\label{df:GTCC}
Let $\eta$ be a timelike ray and $x \in I(\eta)$.
We say that the \emph{generalized timelike co-ray condition}
(\emph{GTCC} for short) for $\eta$ holds at $x$
if any generalized co-ray of $\eta$ emanating from $x$ is timelike.
\end{definition}

Under the GTCC, we can prove the following existence result of maximizing timelike geodesics
(by using Proposition~\ref{pr:GH3.1} and Lemma~\ref{lm:GH2.4}).

\begin{lemma}\label{lm:GH3.3}
Let $\eta:[0,\infty) \lra M$ be a timelike ray and suppose that the GTCC for $\eta$
holds at some $x \in I(\eta)$.
Then there exist a neighborhood $U$ of $x$ and $T>0$ such that,
for any $z \in U$ and $t>T$, there exists a maximizing timelike geodesic from $z$ to $\eta(t)$.
\end{lemma}

See \cite[Lemma~3.3]{GH} for a proof
(roughly speaking,
if there do not exist such $U$ and $T$, then we find a lightlike generalized co-ray of $\eta$
emanating from $x$, a contradiction to the GTCC).
A similar argument shows that the GTCC is an open condition in $I(\eta)$.

Furthermore, taking smaller $U$ and larger $T$ if necessary,
we can show that there exists a compact set $K \subset \Omega \cap F^{-1}(1)$
such that any maximizing timelike geodesic $\zeta$ from $z \in U$ to $\eta(t)$ with $t>T$
of unit speed satisfies $\dot{\zeta}(0) \in K$ (we refer to \cite[Lemma~3.4]{GH} for a proof).
Therefore, we find the following.

\begin{lemma}[Existence of timelike asymptotes]\label{lm:asymp}
Let $\eta:[0,\infty) \lra M$ be a timelike ray and suppose that the GTCC for $\eta$
holds at some $x \in I(\eta)$.
Then there exists a neighborhood $U$ of $x$ such that every $z \in U$
admits a timelike asymptote of $\eta$ emanating from $z$.
\end{lemma}

Finally, the GTCC holds on a neighborhood of $\eta$ as follows.

\begin{proposition}[GTCC near rays]\label{pr:GH5.1}
Let $\eta:[0,\infty) \lra M$ be a timelike ray.
Then, any generalized co-ray emanating from $\eta(a)$ with $a>0$
necessarily coincides with $\eta$.
In particular, the GTCC holds on an open set including $\eta((0,\infty))$.
\end{proposition}

The first assertion is shown in the same manner as \cite[Proposition~5.1]{GH}
(and could be compared with \cite[Lemma~3.4]{Es}).
The second assertion follows from the fact that the GTCC is an open condition
(see \cite[Corollary~5.2]{GH}).

\section{Busemann functions}\label{sc:busemann}

In this section, we continue assuming the future timelike geodesic completeness.

\subsection{For rays}\label{ssc:rays}

Let $\eta:[0,\infty)\lra M$ be a timelike ray of unit speed.
We introduce a central ingredient of the proof of the splitting theorem.

\begin{definition}[Busemann functions]\label{df:Buse}
Define the \emph{Busemann function} $\bb^{\eta}:M \lra [-\infty,+\infty]$
associated with $\eta$ by
\[ \bb^{\eta}(x) :=\lim_{t \to \infty} \bb^{\eta}_t(x), \qquad
 \text{where}\,\ \bb^{\eta}_t(x) :=t-d\big( x,\eta(t) \big). \]
\end{definition}

The limit above always exists in $\R \cup \{\pm \infty\}$.
Precisely, if $x \not\in I^-(\eta)$, then $d(x,\eta(t))=0$ for all $t$ and hence $\bb^{\eta}(x)=\infty$
(recall \eqref{eq:I(eta)} for the definition of $I^-(\eta)$).
If $x \in I^-(\eta)$, then the reverse triangle inequality implies that, for large $s<t$,
\[ \bb^{\eta}_t(x) \le t-\big\{ d\big( x,\eta(s) \big) +d\big( \eta(s),\eta(t) \big) \big\}
 =\bb^{\eta}_s(x). \]
Hence $\bb^{\eta}_t(x)$ is non-increasing in $t$ and converges to $\bb^{\eta}(x) \in \R \cup \{-\infty\}$.
Moreover, since $d$ is lower semi-continuous, $\bb^{\eta}$ is upper semi-continuous on $I^-(\eta)$.

When $x \in I(\eta)$, it follows from \eqref{eq:GH2-4} (with $y=x$) that
$\bb^{\eta}(x) \ge d(\eta(0),x) \ge 0$.
Therefore $\bb^{\eta}$ is well-posed in the domain $I(\eta)$ of our interest.
We also observe that, for $x,y \in M$ with $x \le y$,
\begin{equation}\label{eq:GH2-6}
\bb^{\eta}(y) \ge \bb^{\eta}(x) +d(x,y)
\end{equation}
by the reverse triangle inequality.
In the Riemannian or Finsler setting, the reverse inequality to \eqref{eq:GH2-6} holds for all $x,y$
and it implies that the Busemann function is $1$-Lipschitz continuous.
In the Lorentzian case, however, $\bb^{\eta}$ is not necessarily continuous.

Next we introduce useful upper support functions for $\bb^{\eta}$ on $I(\eta)$.
We give a proof (in the same manner as \cite[Lemma~2.5]{GH}) for completeness.

\begin{lemma}[Support functions for $\bb^{\eta}$]\label{lm:GH2.5}
Let $\zeta:[0,\infty) \lra M$ be a timelike asymptote of $\eta$ of unit speed
with $z:=\zeta(0) \in I(\eta)$.
\begin{enumerate}[{\rm (i)}]
\item\label{zeta1}
For each $t>0$,
\begin{equation}\label{eq:GH2-7}
\rho(x) :=\bb^{\eta}(z) +\bb^{\zeta}_t(x) =\bb^{\eta}(z) +t -d\big( x,\zeta(t) \big)
\end{equation}
is an \emph{upper support function} for $\bb^{\eta}$ at $z$ in the sense that
$\rho(x) \ge \bb^{\eta}(x)$ on a neighborhood of $z$ and $\rho(z)=\bb^{\eta}(z)$.

\item\label{zeta2}
We have
\begin{equation}\label{eq:bzeta}
\bb^{\eta} \big( \zeta(t) \big) =\bb^{\eta}(z) +t
\end{equation}
for all $t \ge 0$.
In particular, $I(\zeta) \subset I(\eta)$ holds.
\end{enumerate}
\end{lemma}

Recall from Lemma~\ref{lm:asymp} and Proposition~\ref{pr:GH5.1} that
points in a neighborhood of $\eta((0,\infty))$ admit timelike asymptotes.

\begin{proof}
\eqref{zeta1}
Let $\zeta_k:[0,a_k] \lra M$ be a sequence of maximizing timelike geodesics
with $\zeta_k(0)=z$ converging to $\zeta$.
By definition (Definition~\ref{df:coray}), we have $\zeta_k(a_k)=\eta(t_k)$ for some $t_k$
and $\lim_{k \to \infty} a_k=\lim_{k \to \infty} t_k =\infty$.
Fix $x \in I^+(\eta(0)) \cap I^-(\zeta(t))$ and take large $k$
so that $z \in I^-(\eta(t_k))$, $x \in I^-(\zeta_k(t))$ and $a_k>t$.
Then we deduce from the reverse triangle inequality that
\begin{align*}
\bb^{\eta}_{t_k}(x) -\bb^{\eta}_{t_k}(z)
&= d\big( z,\eta(t_k) \big) -d\big( x,\eta(t_k) \big) \\
&\le d\big( z,\eta(t_k) \big) -d\big( x,\zeta_k(t) \big) -d\big( \zeta_k(t),\eta(t_k) \big) \\
&= d\big( z,\zeta_k(t) \big) -d\big( x,\zeta_k(t) \big).
\end{align*}
Letting $k \to \infty$ yields
\[ \bb^{\eta}(x) -\bb^{\eta}(z) \le d\big( z,\zeta(t) \big) -d\big( x,\zeta(t) \big)
 =t-d\big( x,\zeta(t) \big) \]
with the help of Lemma~\ref{lm:GH2.3}.
This completes the proof of \eqref{zeta1}.

\eqref{zeta2}
For $s<t$, we have $\zeta(s) \in I^+(\eta(0)) \cap I^-(\zeta(t))$
and hence \eqref{zeta1} implies
\[ \bb^{\eta} \big( \zeta(s) \big) \le \rho \big( \zeta(s) \big) =\bb^{\eta}(z) +s. \]
Combining this with \eqref{eq:GH2-6},
we find that $\bb^{\eta}(\zeta(s)) =\bb^{\eta}(z) +s$ holds for all $s \ge 0$.
This also shows that $\bb^{\eta}<\infty$ on $\zeta$ and hence $\zeta([0,\infty)) \subset I(\eta)$.
Thus we have $I(\zeta) \subset I(\eta)$.
$\qedd$
\end{proof}

The continuity of the Busemann function can be shown
in the same way as \cite[Theorem~3.7]{GH} (see also \cite[Lemma~3.3]{Es}).

\begin{theorem}[Continuity of $\bb^{\eta}$]\label{th:GH3.7}
Assume that the GTCC for $\eta$ holds at $x \in I(\eta)$.
Then $\bb^{\eta}$ is Lipschitz continuous on a neighborhood of $x$.
\end{theorem}

We remark that the Lipschitz continuity is understood with respect to an auxiliary Riemannian metric
(the choice of a metric costs no generality since the assertion is local).

For the sake of further analyzing support functions as in Lemma~\ref{lm:GH2.5},
we would like to have some control over the past timelike cut locus.
Given $z \in M$, we say that $x \in I^-(z)$ is a \emph{past timelike cut point} to $z$
if there is a maximizing timelike geodesic $\xi:[0,1] \lra M$ from $x$ to $z$
such that its extension $\tilde{\xi}:[-\delta,1] \lra M$ is not maximizing for any $\delta>0$.
The \emph{past timelike cut locus} of $z$ is the set of all past timelike cut points to $z$.
(We remark that the past timelike cut locus for $L$ is the future timelike cut locus for $\rev{L}$.)

\begin{proposition}\label{pr:GH3.9}
Let $\zeta:[0,\infty) \lra M$ be a timelike asymptote of $\eta$ with $\zeta(0) \in I(\eta)$
and assume that the GTCC for $\eta$ holds at $\zeta(0)$.
Then, for each $t>0$, there exists a neighborhood of $\zeta([0,t])$
which does not meet the past timelike cut locus of $\zeta(t)$.
\end{proposition}

This is proved along the lines of \cite[Proposition~3.9]{GH}.
We remark that a \emph{past-directed} causal curve means a future-directed causal curve
with respect to $\rev{L}$ in our setting.
(To be precise, Proposition~\ref{pr:GH3.9} is an analogue of \cite[Proposition~14.23]{BEE}.
There is a small gap in the proof of \cite[Lemma~3.10]{GH}
which can be fixed assuming the continuity of $d(\cdot,\alpha(r))$ at $p$,
and the continuity of $d$ in turn follows from the continuity of the Busemann function
(Theorem~\ref{th:GH3.7}) as proved in \cite[Sublemma~14.24]{BEE}.)

It follows from Proposition~\ref{pr:GH3.9} that $d(\cdot,\zeta(t))$ is smooth near $\zeta(0)$
and the Laplacian comparison theorem (Corollary~\ref{cr:Lcomp})
can be used to analyze $\rho(x)$ as in \eqref{eq:GH2-7}.

\subsection{For straight lines}\label{ssc:lines}

In splitting theorems, we assume the existence of a timelike \emph{straight line} $\eta:\R \lra M$,
which is a globally maximizing timelike geodesic of unit speed,
i.e., $d(\eta(s),\eta(t))=t-s$ holds for all $s<t$.
We will denote by $\bb^{\eta}$ the Busemann function associated with the ray $\eta|_{[0,\infty)}$.
Moreover, the curve $\bar{\eta}(t):=\eta(-t)$, $t \in [0,\infty)$, is a timelike ray of unit speed
with respect to the reverse structure $\rev{L}(v)=L(-v)$ (recall Section~\ref{sc:wLF}).
Hence we can introduce the corresponding Busemann function as
\[ \ol{\bb}{}^{\eta}(x):=\lim_{t \to \infty} \big\{ t-d\big( \eta(-t),x \big) \big\} \]
($d$ is with respect to $L$).
It follows from the reverse triangle inequality that
\begin{equation}\label{eq:b+b}
\bb^{\eta}(x)+\ol{\bb}{}^{\eta}(x) \ge \lim_{t \to \infty} \big\{ 2t -d\big( \eta(-t),\eta(t) \big) \big\} =0,
\end{equation}
and equality holds on $\eta$.

For straight lines we modify the definition of \eqref{eq:I(eta)} into
\[ I(\eta) :=\Bigg( \bigcup_{t>0} I^+\big( \eta(-t) \big) \Bigg)
 \cap \Bigg( \bigcup_{t>0} I^- \big( \eta(t) \big) \Bigg). \]
Note that $I(\eta) =\bigcup_{s<0} I(\eta|_{[s,\infty)})$,
therefore the results in this section are available in $I(\eta)$.
We deduce from \eqref{eq:GH2-6} and \eqref{eq:bzeta} the following (see \cite[Lemma~4.1]{Es}).

\begin{lemma}\label{lm:line}
Assume that $x \in I(\eta)$ satisfies $\bb^{\eta}(x)+\ol{\bb}{}^{\eta}(x)=0$,
and let $\zeta^+$ and $\zeta^-$ be timelike asymptotes of $\eta$ and $\bar{\eta}$
emanating from $x$, respectively.
Then the concatenation of $\zeta^+$ and $\zeta^-$ provides a timelike straight line.
\end{lemma}

Precisely, $\zeta^-$ is an asymptote of $\bar{\eta}$ with respect to $\rev{L}$,
and the concatenation $\zeta$ is given by $\zeta(t):=\zeta^+(t)$ for $t \ge 0$
and $\zeta(t):=\zeta^-(-t)$ for $t<0$.
We say that such a straight line $\zeta$ is \emph{bi-asymptotic} to $\eta$.
We also observe from Lemma~\ref{lm:GH2.5} that
$\bb^{\zeta} +\ol{\bb}{}^{\zeta} \ge \bb^{\eta} +\ol{\bb}{}^{\eta}$ holds on $I(\zeta) \subset I(\eta)$
with equality on $\zeta$.

\section{A key step}\label{sc:max}

With the analysis of rays and Busemann functions in the previous sections in hand,
we shall generalize Eschenburg's argument \cite{Es} toward the splitting theorem.
Proposition~\ref{pr:max} below, which generalizes \cite[Proposition~6.1]{Es}, is a crucial step.
Recall that the spacetime Laplacian is not elliptic but hyperbolic.
Therefore we need to consider a spacelike hypersurface and the Laplacian restricted on it,
which is elliptic.
The metrizability will be assumed in Proposition~\ref{pr:max}
to overcome a difficulty arising in our Lorentz--Finsler situation.
For the sake of clarifying where we need the metrizability,
we do not intend to utilize it fully but minimize the use of it in our discussion.

\begin{remark}[Positive-definite case]\label{rm:posi}
In the positive-definite case,
the Laplacian comparison theorem implies the subharmonicity of the Busemann function
$\bb^{\eta}$, namely $\Lap \bb^{\eta} \ge 0$.
We similarly obtain the subharmonicity of $\ol{\bb}{}^{\eta}$
with respect to the reverse Finsler structure ($\rev{\Lap} \ol{\bb}{}^{\eta} \ge 0$),
and hence $\Lap \bb^{\eta} \ge -\rev{\Lap} \ol{\bb}{}^{\eta} =\Lap(-\ol{\bb}{}^{\eta})$.
Comparing this with $\bb^{\eta} +\ol{\bb}{}^{\eta} \le 0$ by the triangle inequality
and $\bb^{\eta} +\ol{\bb}{}^{\eta} \equiv 0$ on $\eta$,
we deduce from the strong comparison principle that
$\bb^{\eta} +\ol{\bb}{}^{\eta} \equiv 0$
as well as $\Lap \bb^{\eta} =\rev{\Lap} \ol{\bb}{}^{\eta} \equiv 0$.
We refer to \cite{Osplit} for details.
\end{remark}

We begin with a bound of the Hessian of the distance function.
Henceforth, we assume the Berwald condition.
Recall \eqref{eq:Hess} for the definition of the Hessian,
and Definition~\ref{df:Ber} for Berwald spacetimes.

\begin{lemma}[Hessian bound]\label{lm:GH4.3}
Let $(M,L)$ be a future timelike geodesically complete Berwald spacetime,
$\eta:[0,\infty) \lra M$ be a timelike ray of unit speed,
and take $x \in I(\eta)$ at where the GTCC for $\eta$ holds.
Then, given a small neighborhood $U$ of $x$ and $t_0>0$,
there exists a constant $C=C(U,t_0)>0$ such that,
for each timelike asymptote $\zeta:[0,\infty) \lra M$ of $\eta$ of unit speed
with $z:=\zeta(0) \in U$, we have
\begin{equation}\label{eq:GH4-1}
g_{\Grad d_t(z)} \big( \Grad^2 d_t(v),v \big) \ge -C \cdot g_{\dot{\zeta}(0)}(v^{\perp},v^{\perp})
\end{equation}
for all $v \in T_zM$ and $t \ge t_0$, where $d_t:=d(\cdot,\zeta(t))$ and
$v^\perp$ is the projection of $v$ to the $g_{\dot{\zeta}(0)}$-normal space to $\dot{\zeta}(0)$
in $T_zM$.
\end{lemma}

We take $U$ on which the GTCC for $\eta$ holds.
Note that $-d_t$ is a temporal function and $C^{\infty}$ on a neighborhood of $z$
thanks to Proposition~\ref{pr:GH3.9}, thereby $\Grad^2 d_t(v)$ is well-defined.
Observe also that $g_{\dot{\zeta}(0)}(v^{\perp},v^{\perp}) \ge 0$ and
$v^{\perp}=v+g_{\dot{\zeta}(0)}(\dot{\zeta}(0),v)\dot{\zeta}(0)$
since
\[ g_{\dot{\zeta}(0)}\Big( \dot{\zeta}(0),v+g_{\dot{\zeta}(0)}\big( \dot{\zeta}(0),v \big) \dot{\zeta}(0) \Big)
 =g_{\dot{\zeta}(0)}\big( \dot{\zeta}(0),v \big)
 \big\{ 1+g_{\dot{\zeta}(0)} \big( \dot{\zeta}(0),\dot{\zeta}(0) \big) \big\} =0 \]
by $g_{\dot{\zeta}(0)}(\dot{\zeta}(0),\dot{\zeta}(0)) =2L(\dot{\zeta}(0)) =-1$.

\begin{proof}
Note that $(v,w) \longmapsto g_{\Grad d_t(z)} (\Grad^2 d_t(v),w)$ is a quadratic form
by the symmetry \eqref{eq:symm} and that
\[ \Grad^2 d_t \big( \dot{\zeta}(0) \big) =D^{ \dot{\zeta}}_{ \dot{\zeta}} \dot{\zeta}(0) =0 \]
since $\Grad d_t(\zeta(s))=\dot{\zeta}(s)$ for $s \in [0,t)$.
Hence it is sufficient to show \eqref{eq:GH4-1} for all $v$ in the
$g_{\dot{\zeta}(0)}$-normal space to $\dot{\zeta}(0)$ (i.e., $v=v^{\perp}$).

As we mentioned before Lemma~\ref{lm:asymp},
the set of initial vectors of timelike asymptotes of unit speed
emanating from a (small) neighborhood $U$
is included in a compact set of unit timelike vectors (see \cite[Corollary~3.5]{GH} as well).
Therefore, given $t_0>0$, thanks to the smoothness of $d_{t_0}$,
we find $C>0$ such that \eqref{eq:GH4-1} holds at $t=t_0$.

For extending \eqref{eq:GH4-1} to $t>t_0$,
it suffices to show that $g_{\Grad d_t(z)} (\Grad^2 d_t(v),v)$ is non-decreasing in $t$.
To this end, recall that
\begin{equation}\label{eq:Hd}
g_{\Grad d_t(z)} \big( \Grad^2 d_t(v),v \big) =(d_t \circ \xi_v)''(0)
\end{equation}
holds for the geodesic $\xi_v:(-\delta,\delta) \lra M$ with $\dot{\xi}_v(0)=v$
(see Remark~\ref{rm:Hess}).
Observe from the reverse triangle inequality that
\[ d_t \big( \xi_v(r) \big) \ge d_s \big( \xi_v(r) \big) +d\big( \zeta(s),\zeta(t) \big)
 =d_s \big( \xi_v(r) \big) +t-s \]
for small $r>0$ and $t_0 \le s<t$, while $d_t(z)=d_s(z)+t-s$.
Combining these with
\[ (d_t \circ \xi_v)'(0) =(d_s \circ \xi_v)'(0) =g_{\dot{\zeta}(0)} \big( \dot{\zeta}(0),v \big) \]
from the first variation formula,
we find that $d_t$ is more convex than $d_s$ at $z$.
Therefore we have
\[ (d_t \circ \xi_v)''(0) \ge (d_s \circ \xi_v)''(0), \]
which completes the proof by \eqref{eq:Hd}.
$\qedd$
\end{proof}

We are ready to show the main result in this section.
Proposition~\ref{pr:GH5.1} will play a crucial role (in place of \cite[Lemma~3.4]{Es}).

\begin{proposition}[Key step]\label{pr:max}
Let $(M,L,\Psi)$ be a weighted, timelike geodesically complete Berwald spacetime
of $\Ric_N \ge 0$ on $\Omega$ for some $N \in (-\infty,0) \cup [n,+\infty]$.
Suppose also that
\begin{enumerate}[$(1)$]
\item\label{max-cplt}
the $\ez$-completeness for some $\ez$ in the corresponding $\ez$-range \eqref{eq:erange} holds,
\item\label{max-line}
there is a timelike straight line $\eta:\R \lra M$,
\item\label{max-met}
there is a timelike geodesically complete Lorentzian structure $g$ of $M$ whose Levi-Civita connection
coincides with the Chern connection of $L$ such that $\eta$ is timelike for $g$,
\item\label{max-cover}
in the universal cover $(\wt{M},\tilde{g})$ of $(M,g)$,
a lift of $\eta$ is a timelike straight line.
\end{enumerate}
Then there exists a neighborhood $W$ of $\eta(\R)$ on which we have
$\bb^{\eta}+\ol{\bb}{}^{\eta}=0$.
\end{proposition}

\begin{proof}
We will suppress $\eta$ in $\bb^{\eta}$ and $\ol{\bb}{}^{\eta}$.
The proof is divided into five steps.
We follow the lines of \cite[Proposition~6.1]{Es} as far as possible,
and the metrizability \eqref{max-met} comes into play in the fourth step.

\begin{step}\label{step1}
Observe first that $\bb$ and $\ol{\bb}$ are continuous on a neighborhood $W_0$ of $\eta(\R)$
by Theorem~\ref{th:GH3.7} and Proposition~\ref{pr:GH5.1}.
Now, for deriving a contradiction, suppose that there is no neighborhood $W$ satisfying the claim.
Then we find that the open set
\[ Q:=\{ x \in W_0 \,|\, \bb(x)+\ol{\bb}(x)>0 \} \]
has a boundary point on $\eta(\R)$, denoted by $x_0$
(recall that $\bb+\ol{\bb} \ge 0$ by \eqref{eq:b+b}).
We may assume by translation that $\eta(0)=x_0$.
Moreover, we can take a small coordinate domain $B$ in $Q$ such that $x_0 \in \del B$
(possibly replacing $\eta$ with a timelike straight line given by Lemma~\ref{lm:line};
see \cite[Proposition~6.1]{Es} for details).

We take Fermi(--Walker) coordinates $(x^{\alpha})_{\alpha=0}^n$
on a sufficiently small tubular neighborhood $U$ around $\eta([-R,R])$,
where $R \in (0,1/3]$. 
Precisely, given $g_{\dot{\eta}}$-parallel vector fields
$P_{\alpha}$ ($\alpha=0,1,\ldots,n$) along $\eta$ such that
$\{P_{\alpha}(0)\}_{\alpha=0}^n$ is $g_{\dot{\eta}}$-orthonormal with $P_0=\dot{\eta}$,
we consider a local coordinate system given by
\[ (x^{\alpha})_{\alpha=0}^n \,\longmapsto\,
 \exp_{\eta(x^0)} \Bigg( \sum_{i=1}^n x^i P_i(x^0) \Bigg),
 \qquad |x^0|<R,\,\ \sum_{i=1}^n (x^i)^2 <R^2. \]
We remark that the exponential map is $C^{\infty}$ on a neighborhood of the zero section
by the Berwald condition.

By taking smaller $B$ if necessary, $B \cap U$ may be represented as
\[ B \cap U =\Bigg\{ (x^{\alpha})_{\alpha=0}^n \in U
 \,\Bigg|\, \sum_{\alpha=0}^n (x^{\alpha})^2 -6R x^n <0 \Bigg\} \]
in the coordinates above. 
Now we consider the set 
\[ \cA :=\Bigg\{ (x^\alpha)_{\alpha=0}^n \in U \,\Bigg|\,
 \sum_{i=1}^n (x^i)^2 > (1+2R)(x^0)^2 \Bigg\} \]
including the spacelike hypersurface $\{x^0=0\} \setminus \{x_0\}$.
Note that the Fermi coordinates enjoy $\Gamma^{\alpha}_{\beta \delta}(\dot{\eta})=0$ along $\eta$
(see, e.g., \cite[Proposition~3.3]{Min-coord} in a more general non-Berwald situation).
\end{step}

\begin{step}\label{step2}
We consider the flat Minkowski structure
\[ g_0 = -(\td x^0)^2 +(\td x^1)^2 +\cdots +(\td x^n)^2 \]
associated with the Fermi coordinates $(x^{\alpha})_{\alpha=0}^n$ above.
Note that $g_0=g_{\del/\del x^0}$ along $\eta$.

We can construct a smooth function $\vartheta$ on $U$ satisfying the following properties:
\begin{enumerate}[(a)]
\item\label{h-a} $\vartheta(x_0)=0$;
\item\label{h-b} $\vartheta >0$ on $\cA \setminus B$.
\end{enumerate}
This is done along the lines of \cite[Lemma~6.2]{Es} as follows.
We first consider a function
\[ \phi(x)=\frac{1}{R} \sum_{i=1}^{n-1} (x^i)^2 -x^n -(x^0)^2 \]
in the above coordinates,
and define $\vartheta:=1-\e^{-\sigma \phi}$ for large $\sigma>0$.
Then \eqref{h-a} is clear and \eqref{h-b} is seen by
the argument of Claim in the proof of \cite[Lemma~6.2]{Es} as follows.

It suffices to show that $x \in \cA$ satisfying $\phi(x) \le 0$ necessarily lies in $B$.
We deduce from $x \in \cA$ and $\phi(x) \le 0$ that
\[ (1+2R)(x^0)^2 <R\big( x^n +(x^0)^2 \big) +(x^n)^2, \]
which implies
\[ \bigg( x^n +\frac{R}{2} \bigg)^2 -\frac{R^2}{4} >(1+R)(x^0)^2 \ge 0. \]
Noticing $x^n >-R$, we have $x^n>0$ and
\[ x^n > \sqrt{(1+R)(x^0)^2 +\frac{R^2}{4}} -\frac{R}{2}
 =\frac{R}{2} \Bigg( \sqrt{\frac{4(1+R)}{R^2}(x^0)^2 +1} -1 \Bigg). \]
Since $4(1+R)(x^0)^2/R^2 <4(1+R)<8$
and the slope of the function $\sqrt{s}$ is larger than $1/6$ for $s \in (1,9)$,
we observe
\[ \sqrt{\frac{4(1+R)}{R^2}(x^0)^2 +1} > 1+\frac{1}{6} \frac{4(1+R)}{R^2}(x^0)^2. \]
Thus we have
\[ x^n >\frac{1+R}{3R} (x^0)^2,  \]
and hence $x \in B$ since
\[ \sum_{\alpha=0}^n (x^{\alpha})^2 -6R x^n
 \le (1+R)(x^0)^2 +(x^n)^2 -5Rx^n
 < (x^n)^2 -2Rx^n <0. \]
\end{step}

\begin{step}\label{step3}
We put
\[ U_r :=\Bigg\{ \exp_{\eta(x^0)} \Bigg( \sum_{i=1}^n x^i P_i(x^0) \Bigg) \,\Bigg|\,
 |x^0|<R,\ \sum_{i=1}^n (x^i)^2 <r^2 \Bigg\} \subset U \]
for $r \in (0,R]$ and $S_r:=(\del U_r) \cap \{\bb \le 0\}$.
Then $\ol{\bb} \ge 0$ on $S_r$ by \eqref{eq:b+b}.
Moreover, taking smaller $r$ if necessary, we have
\[ S_r \subset \{\ol{\bb}>0\} \cup \{ \vartheta>0 \} \]
since $\bb<0$ on $S_r \setminus \cA$ (by the first variation formula for $\bb^{\eta}_t$ for large $t$),
$\vartheta>0$ on $\cA \setminus B$ by \eqref{h-b},
and $\ol{\bb}>0$ on $B \cap \{\bb \le 0\}$.

This implies that $\inf_{S_r \cap \{\vartheta \le 0\}} \ol{\bb}>0$, therefore,
for small $\tau>0$ (depending on $r$), $f:=\ol{\bb} +\tau \vartheta$ is positive on $S_r$.
Set $U^-_r :=U_r \cap \{\bb \le 0\}$ and observe that
\[ \del U^-_r \subset S_r \cup \Big( \overset{\circ}{U}_r \cap \bb^{-1}(0) \Big). \]
Since $f(x_0)=0$ by \eqref{h-a} and $f|_{S_r} >0$, $f|_{U^-_r}$ takes its minimum value at some point
\[ z \in \overset{\circ}{U}{}^-_r \cup \Big( \overset{\circ}{U}_r \cap \bb^{-1}(0) \Big). \]
In fact, by construction, we find $\td f(z) \neq 0$ (for small $\tau$)
and hence $z \in \overset{\circ}{U}_r \cap \bb^{-1}(0)$.

For large $s>0$,
let $\zeta^-:[0,\infty) \lra M$ be an asymptote of $\bar{\eta}(t)=\eta(-t)$ of unit speed
emanating from $z$ with respect to $\rev{L}$, and we consider
\[ \bar{\rho}(x) :=\ol{\bb}(z) +s -d\big( \zeta^-(s),x \big), \]
which is an upper support function for $\ol{\bb}$ at $z$ in the same manner as \eqref{eq:GH2-7}.
Then,
\[ f_s:=\bar{\rho} +\tau \vartheta \]
is an upper support function for $f$ at $z$
and $\min_{U^-_r} f_s$ is attained again at $z$.

Next, let $\zeta^+$ be an asymptote of $\eta|_{[0,\infty)}$ of unit speed
with $\zeta^+(0)=z$ and put
\[ \rho(x) :=s -d\big( x,\zeta^+(s) \big), \]
which is an upper support function for $\bb$ at $z$ by Lemma~\ref{lm:GH2.5}.
Then the level surface $H_s:=\rho^{-1}(0) \cap U_r$ passing through $z$
is a smooth $g_{\Grad(-\rho)}$-spacelike hypersurface.
Since $\bb \le \rho =0$ on $H_s$, we have $H_s \subset U^-_r$.
Hence $f_s|_{H_s}$ takes its minimum at $z$ and, since $\td f_s(z) \neq 0$,
we have $\td f_s (z) =-\Lambda \,\td\rho(z)$ for some $\Lambda >0$ close to $1$
(both are close to $-\td x^0$ by virtue of Proposition~\ref{pr:GH5.1}).
\end{step}

\begin{step}\label{step4}
Now, we make use of the metrizability \eqref{max-met}.
Note that geodesics are common between $L$ and $g$.
In this step we consider the case where $\eta$ is a timelike straight line also for $g$.
Since asymptotes of $\eta$ with respect to $g$
has tangent vectors close to $\dot{\eta}$ by Proposition~\ref{pr:GH5.1},
they are timelike also for $L$ and $\Ric_N \ge 0$ holds along them in $(M,g,\Psi)$.
Thus we can apply the splitting theorem in \cite{WW2}
to obtain a local isometric splitting
$(\widehat{W},g)=(\R \times \widehat{\Sigma},-\td t^2 +\hat{h})$
of a neighborhood $\widehat{W}$ of $\eta$,
and $\Psi$ is constant on each line $t \longmapsto (t,x)$.
Note that the geodesic equation on $\widehat{W}$ is decomposed into
those on $\R$ and $\widehat{\Sigma}$.

Since parallel transports preserve $L$ in Berwald spacetimes
(see, e.g., \cite[Proposition~10.1.1]{BCS}),
$(\widehat{W},L)$ inherits the product structure of $(\widehat{W},g)$
(whereas $\widehat{\Sigma}$ may not be spacelike for $g_{\dot{\eta}}$ or $L$).
Precisely, taking smaller $\widehat{W}$ including $\eta$ if necessary,
we have a modified splitting $\widehat{W}=\R \times \Sigma$
such that the map $(s,x) \longmapsto (s+t,x)$ preserves $L$
for all $(s,x) \in \R \times \Sigma$ and $t \in \R$
and that $\Sigma$ is spacelike and perpendicular to $V$ with respect to $g_V$,
where $V$ is the parallel vector field extending $\dot{\eta}$.

In this product structure,
we find that the asymptotes $\zeta^-$ and $\zeta^+$ (for $L$) together
form a straight line parallel to $\eta$.
This implies that $\dot{\zeta}^+(0)=-\dot{\zeta}^-(0)$, and hence
\begin{equation}\label{eq:rho}
\td\bar{\rho}(z) =-\td\rho(z).
\end{equation}
Then we have
\[ \tau \,\td\vartheta(z) =\td f_s(z) -\td\bar{\rho}(z) =(1-\Lambda) \td\rho(z). \]
This is, however, a contradiction since $\td\vartheta(z)$ is close to
$-\sigma \,\td x^n$ by construction.
Therefore, $\bb+\ol{\bb}=0$ holds on a neighborhood of $\eta$.
\end{step}

\begin{step}\label{step5}
Finally, we consider the remaining case where $\eta$ is not a straight line for $g$.
Then we need the hypothesis \eqref{max-cover}.
Denote by $\wt{L}$ the Lorentz--Finsler structure of $\wt{M}$ induced from $L$.
Since a lift $\tilde{\eta}:\R \lra \wt{M}$ of $\eta$ is a timelike straight line for $\tilde{g}$
by \eqref{max-cover},
we can split $(\wt{M},\tilde{g},\wt{\Psi})$ and the Busemann functions
$\bb^{\tilde{\eta}}$ and $\ol{\bb}{}^{\tilde{\eta}}$
with respect to $\wt{L}$ satisfies $\bb^{\tilde{\eta}}+\ol{\bb}{}^{\tilde{\eta}}=0$
on a neighborhood of $\tilde{\eta}$ by the above argument.

Let $\pi:\wt{M} \lra M$ be the projection and observe that
\[ d_{\wt{L}} \big( p,\tilde{\eta}(t) \big) \le d_L\big( \pi(p),\eta(t) \big) \]
for $p \in \wt{M}$ close to $\tilde{\eta}$ and large $t>0$.
Therefore we obtain
\[ 0 \le (\bb+\ol{\bb}) \circ \pi \le \bb^{\tilde{\eta}}+\ol{\bb}{}^{\tilde{\eta}}=0, \]
which implies that $\bb+\ol{\bb}=0$ holds on a neighborhood of $\eta$.
This completes the proof.
$\qedd$
\end{step}
\end{proof}

A typical situation where $\eta$ is not a straight line for $g$ is that $M$ is a cylinder:
\[ (M,g)=\big( \R \times (\R/\Z),-\td t^2 + \td s^2 \big), \qquad \eta(t)=(t,ct) \]
for small $c>0$.
By modifying $g$, we can construct a flat Lorentzian metric
for which $\eta$ is a timelike straight line.
Two more remarks on the above proof are in order.

\begin{remark}\label{rm:max}
The above proof relying on the splitting theorem of $g$ is obviously not a standard one.
The strategy in \cite{Es} is to show that the Laplacian of $f_s|_{H_s}$ is negative at $z$
by using the Laplacian comparison theorem (Theorem~\ref{th:Lcomp}, Corollary~\ref{cr:Lcomp})
and the Hessian bound (Lemma~\ref{lm:GH4.3}).
One can follow this line to some extent by using the linearized Laplacian
$\Delta\!^{\Grad\bar{\rho}}$ with respect to $g_{\Grad\bar{\rho}}$
and fine properties of the Fermi coordinates.
In general, however, there is a difficulty arising from the absence of \eqref{eq:rho}
(precisely, we do not know how to estimate $\Delta\!^{\Grad\bar{\rho}} \rho$).
This is the reason why we assumed the metrizability
and digressed from \cite{Es} in Steps~\ref{step4} and \ref{step5}.
\end{remark}

\begin{remark}[The case of $\psi_{\fm}$]\label{rm:psi_m}
When we consider the weight function $\psi_{\fm}$ associated with a measure $\fm$ on $M$
as in Remark~\ref{rm:meas}, along any geodesic $\xi$,
$\vol_{g_{\dot{\xi}}}$ is a constant multiplication of $\vol_g$
since parallel transports preserve $L$ in Berwald spacetimes.
Hence $\Ric_N$ for $(M,L,\fm)$ coincides with that for $(M,g,\Psi_{\fm})$,
where $\Psi_{\fm} \in C^{\infty}(M)$ is given by $\fm=\e^{-\Psi_{\fm}} \,\vol_g$,
and we can apply the above proof with $(M,g,\Psi_{\fm})$.
\end{remark}

\section{Proof of the splitting theorem}\label{sc:split}

This last section is devoted to the proof of the main theorem.
We need to recall some standard concepts in Lorentzian geometry
and their counterparts in the weighted setting;
see \cite{LMO,LMO2} for details.

Let $\zeta:[0,l) \lra M$ be a timelike geodesic of unit speed
and denote by $N_{\zeta}(t) \subset T_{\zeta(t)}M$ the $n$-dimensional subspace
$g_{\dot{\zeta}(t)}$-orthogonal to $\dot{\zeta}(t)$.
Given a \emph{Lagrange tensor field} $\sJ(t):N_{\zeta}(t) \lra N_{\zeta}(t)$ along $\zeta$,
we set $\sB:=\sJ' \sJ^{-1}$.
Then we have the \emph{Riccati equation}
\begin{equation}\label{eq:Ricc}
\sB' +\sB^2 +\sR=0,
\end{equation}
where $\sR(t):=R_{\dot{\zeta}(t)}:N_{\zeta}(t) \lra N_{\zeta}(t)$ is the curvature endomorphism.

Taking the weight function $\Psi$ on $M$ into account, we define
\begin{equation}\label{eq:B_e}
\sB_{\ez}(t) :=\e^{\frac{2(1-\ez)}{n} \Psi(\zeta(t))} \bigg(
 \sB(t) -\frac{(\Psi \circ \zeta)'(t)}{n} \sI_n(t) \bigg),
\end{equation}
where $\sI_n(t):N_{\zeta}(t) \lra N_{\zeta}(t)$ is the identity.
We also define
\begin{align*}
\theta_{\ez}(t) &:=\trace\big( \sB_{\ez}(t) \big)
 =\e^{\frac{2(1-\ez)}{n}\Psi(\zeta(t))} \Big\{ \trace\big( \sB(t) \big) -(\Psi \circ \zeta)'(t) \Big\}, \\
 \sigma_{\ez}(t) &:=\sB_{\ez}(t) -\frac{\theta_{\ez}(t)}{n} \sI_n(t)
\end{align*}
as the \emph{expansion scalar} and the \emph{shear tensor}.
Then, for $N \in (-\infty,0) \cup [n,+\infty]$,
it was established in \cite[Theorem~5.6]{LMO}
the following \emph{timelike weighted Raychaudhuri equation}:
\begin{equation}\label{eq:Rayeq}
\theta_{\ez}^* +c\theta_{\ez}^2
 +\frac{N(N-n)}{n} \bigg( \frac{\ez \theta_{\ez}}{N} +\frac{(\Psi \circ \zeta)^*}{N-n} \bigg)^2
 +\trace(\sigma_{\ez}^2) +\Ric_N(\zeta^*) =0,
\end{equation}
where $c=c(N,\ez)>0$ in \eqref{eq:LF-c},
\begin{align*}
\varphi_{\zeta}(t)
&:= \int_0^t \e^{\frac{2(\ez-1)}{n} \Psi(\zeta(s))} \,\td s, \\
\theta_{\ez}^*(t)
&:= (\theta_{\ez} \circ \varphi_{\zeta}^{-1})' \big( \varphi_{\zeta}(t) \big)
 =\e^{\frac{2(1-\ez)}{n} \Psi(\zeta(t))} \theta'_{\ez}(t),
\end{align*}
and similarly
$(\Psi \circ \zeta)^*(t) :=\e^{\frac{2(1-\ez)}{n} \Psi(\zeta(t))} (\Psi \circ \zeta)'(t)$
and $\zeta^*(t):=\e^{\frac{2(1-\ez)}{n} \Psi(\zeta(t))} \dot{\zeta}(t)$.

\subsection{Local splitting}\label{ssc:local}

We first build a local splitting on top of Proposition~\ref{pr:max}.

\begin{proposition}[Local isometric splitting]\label{pr:Lsplit}
Let $(M,L,\Psi)$ be a weighted Berwald spacetime as in Proposition~$\ref{pr:max}$.
Then a neighborhood $W'$ of $\eta(\R)$ isometrically splits
in the sense that there exists an $n$-dimensional Riemannian manifold $(\Sigma,h)$
such that $(W',g_{\Grad \ol{\bb}{}^{\eta}})$ is isometric to $(\R \times \Sigma,-\td t^2 +h)$.

Furthermore, for every $x \in \Sigma$,
the image $\zeta(t)$ of $(t,x)$ is a straight line bi-asymptotic to $\eta$
and $\Psi$ is constant on $\zeta$.
\end{proposition}

The latter assertion is interpreted as the splitting of the weight function $\Psi$.

\begin{proof}
We again suppress $\eta$ in $\bb^{\eta}$ and $\ol{\bb}{}^{\eta}$.
As is common with splitting theorems, we shall show that $\bb$ is affine
(see \cite[Proposition~6.3]{Es} for the Lorentzian case
and \cite[Section~4]{EH} for the Riemannian case),
for which the Berwald condition is essential
(we refer to \cite[Lemma~5.1]{Osplit} for the Finsler case).

Fix $z \in W$ and let $\zeta:\R \lra M$ be the timelike straight line of unit speed
bi-asymptotic to $\eta$ with $\zeta(0)=z$ (given by Lemma~\ref{lm:line}).
Since there is no conjugate point along $\zeta$,
we have $\sR(t)=0$ on $N_{\zeta}(t)$ for any $t \in \R$ by \cite[Proposition~7.6]{LMO}.
Similarly to the proof of Proposition~\ref{pr:max}, we fix (large) $s>0$ and consider
\[ \bar{\rho}_s(x) :=\ol{\bb}(z) +s -d\big( \zeta(-s),x \big), \]
which is an upper support function for $\ol{\bb}$ at $z$.
Now, let $\sJ_s(t):N_{\zeta}(t) \lra N_{\zeta}(t)$ be the Lagrange tensor field along $\zeta$
with $\sJ_s(-s)=0$ and $\sJ'_s(-s)=\sI_n$ (see \cite[Proposition~5.13]{LMO} for a construction).
We set $\sB^s:=\sJ'_s \sJ_s^{-1}$, and observe that
it is the Hessian of the distance function $d_s:=d(\zeta(-s),\cdot)$ as we saw in the proof of
\cite[Theorem~5.8]{LMO2}.
Precisely, we have $[\sB^s(t)](w)=\Grad^2(-d_s)(w)$ for all $w \in N_{\zeta}(t)$.

Next we define
\[ \sB^s_{\ez} :=\e^{\frac{2(1-\ez)}{n} \Psi \circ \zeta} \bigg(
 \sB^s -\frac{(\Psi \circ \zeta)'}{n} \sI_n \bigg), \qquad
 \theta^s_{\ez}:=\trace(\sB^s_{\ez}), \]
similarly to \eqref{eq:B_e}.
Since $\zeta$ contains no conjugate point and is $\ez$-complete,
it follows from \cite[Proposition~5.8]{LMO} that $\theta^s_{\ez} \ge 0$.
Combining this lower bound and the fact that $\sB^s(0)$ is non-increasing in $s$
as in the proof of Lemma~\ref{lm:GH4.3},
we can take the limit $\sB:=\lim_{s \to \infty} \sB^s$
(we remark that only $\sB^s$ depends on $s$ in the definition of $\sB^s_{\ez}$).
Note that $\sB$ is smooth by the bootstrap argument
and satisfies the Riccati equation \eqref{eq:Ricc} with $\sR \equiv 0$,
i.e., $\sB' +\sB^2=0$.
Then we define $\sB_{\ez}$ in the same way as \eqref{eq:B_e}
and observe from the absence of conjugate points and \cite[Proposition~5.8]{LMO} that
$\theta_{\ez} \equiv 0$.
Therefore it follows from \eqref{eq:Rayeq} and $\Ric_N(\dot{\zeta}) \ge 0$ that
$(\Psi \circ \zeta)' \equiv 0$, $\trace(\sigma_{\ez}^2) \equiv 0$
and $\Ric_N(\dot{\zeta}) \equiv 0$.
In particular, we have $\sB \equiv 0$.

We deduce from $\sB \equiv 0$ that the Hessian of $\bar{\rho}_s$ at $z$
can be arbitrarily close to zero as $s$ goes to infinity.
We similarly find that $\rho_s(x):=\bb(z)+s-d(x,\zeta(s))$
is an upper support function for $\bb$ with the same property.
Hence, for any geodesic $\xi:[0,1] \lra W$,
the functions $\bb \circ \xi$ and $\ol{\bb} \circ \xi$
have upper support functions whose second order derivatives arbitrarily close to zero
(here we use the Berwald condition; recall Remark ~\ref{rm:Hess}).
This implies that $\bb \circ \xi$ and $\ol{\bb} \circ \xi$ are concave.
Combining this with $\bb+\ol{\bb}=0$ on $W$,
we find that $\bb \circ \xi$ and $\ol{\bb} \circ \xi$ are affine functions
along every geodesic in $W$.
Therefore $\ol{\bb}|_W$ is smooth and every level set of $\ol{\bb}$ in $W$ is totally geodesic.
Moreover, the gradient vector field $\Grad \ol{\bb}$ is parallel in the sense that
$\Grad^2 \ol{\bb} \equiv 0$.

Now we set
\[ \Sigma :=\ol{\bb}{}^{-1}(0) \cap W =\bb^{-1}(0) \cap W \]
and define a map $\Theta:\R \times \Sigma \lra M$ by $\Theta(t,x):=\zeta_x(t)$,
where $\zeta_x:\R \lra M$ is the timelike straight line of unit speed
bi-asymptotic to $\eta$ with $\zeta_x(0)=x$.
By construction $\zeta_x$ is an integral curve of $\Grad \ol{\bb}$
and satisfies $\bb(\zeta_x(t))=-\ol{\bb}(\zeta_x(t))=t$ for all $t \in \R$ (recall \eqref{eq:bzeta}).
In particular, $\bb+\ol{\bb}=0$ holds on the image of $\Theta$.
Since every integral curve of $\Grad \ol{\bb}$ is a geodesic,
$\Grad \ol{\bb}$ is parallel also with respect to $g_{\Grad \ol{\bb}}$
(see \cite[Proposition~3.2]{LMO}).
Hence one can apply the de Rham--Wu decomposition theorem \cite{Wu}
to see that $\Theta$ is a local isometry,
where $\Sigma$ and $\Theta(\R \times \Sigma)$ are equipped with
the Riemannian metric $g_{\Grad \ol{\bb}}|_{T\Sigma}$
and the Lorentzian metric $g_{\Grad \ol{\bb}}$, respectively.
Furthermore, $\Theta$ is injective since asymptotes of $\eta$ do not intersect.
This completes the proof of the former assertion by setting $W':=\Theta(\R \times \Sigma)$.
The latter assertion $(\Psi \circ \zeta_x)' \equiv 0$ was already seen above.
$\qedd$
\end{proof}

One can see from the above proof the following corollary on the behavior of
the original Lorentz--Finsler structure $L$
(compare this with \cite[Proposition~5.2]{Osplit} in the Finsler setting).

\begin{corollary}[Local isometric translations]\label{cr:Ltrans}
Let $(M,L,\Psi)$ be as in Proposition~$\ref{pr:Lsplit}$.
Then, in the product structure $W'=\R \times \Sigma$, the translations
\[ \Phi_t(s,x):=(s+t,x), \qquad (s,x) \in \R \times \Sigma,\ t \in \R, \]
are isometric transformations of $(W',L)$.
\end{corollary}

\begin{proof}
It suffices to show that $\Phi_t$ preserves $L$.
Given a straight line $\zeta(t)=\Theta(t,x)$ and $v \in T_xM$,
the vector field $V(t):=\td \Phi_t(v)$ along $\zeta$ is parallel with respect to $g_{\Grad \ol{\bb}}$
by Proposition~\ref{pr:Lsplit}.
Since $D_{\dot{\zeta}}V =D_{\dot{\zeta}}^{g_{\Grad \ol{\bb}}}V \equiv 0$
by the fact that all integral curves of $\Grad \ol{\bb}$ are geodesic
and $\dot{\zeta}=\Grad \ol{\bb}(\zeta)$ (see \cite[Proposition~3.2]{LMO}),
we obtain
\[ \frac{\td}{\td t}\big[ L(V) \big]
 =\frac{1}{2} \frac{\td}{\td t}\big[ g_V(V,V) \big]
 =g_V(D_{\dot{\zeta}}V,V) \equiv 0 \]
as desired (see \cite[(3.2)]{LMO} for the second equality).
$\qedd$
\end{proof}

\subsection{Global splitting}\label{ssc:global}

We can extend the local splitting in Proposition~\ref{pr:Lsplit}
by the open-and-closed argument (see \cite[Section~7]{Es} and \cite[\S 14.4]{BEE}).

\begin{theorem}[Splitting theorem]\label{th:Gsplit}
Let $(M,L,\Psi)$ be a connected, timelike geodesically complete,
weighted Berwald spacetime of $\Ric_N \ge 0$ on $\Omega$
with $N \in (-\infty,0) \cup [n,+\infty]$.
Suppose also that
\begin{enumerate}[$(1)$]
\item\label{split-cplt}
the $\ez$-completeness holds for some $\ez$ in the corresponding $\ez$-range \eqref{eq:erange},
\item\label{split-line}
there is a timelike straight line $\eta:\R \lra M$,
\item\label{split-met}
there is a timelike geodesically complete Lorentzian structure $g$ of $M$
whose Levi-Civita connection coincides with
the Chern connection of $L$ such that $\eta$ is timelike for $g$,
\item\label{split-cover}
in the universal cover $(\wt{M},\tilde{g})$ of $(M,g)$,
a lift of $\eta$ is a timelike straight line.
\end{enumerate}
Then the Lorentzian manifold $(M,g_{\Grad \ol{\bb}{}^{\eta}})$ isometrically splits
in the sense that there exists an $n$-dimensional complete Riemannian manifold $(\Sigma,h)$
such that $(M,g_{\Grad \ol{\bb}{}^{\eta}})$ is isometric to $(\R \times \Sigma,-\td t^2 +h)$.

Furthermore, for every $x \in \Sigma$,
the image $\zeta(t)$ of $(t,x)$ is a straight line bi-asymptotic to $\eta$
and $\Psi$ is constant on $\zeta$.
\end{theorem}

\begin{proof}
We put $Z:=(\bb^{\eta})^{-1}(0)$ for brevity.
We shall show that the local splitting as in Proposition~\ref{pr:Lsplit}
can be extended to cover whole $M$.
To this end, we consider the supremum $r_0$ of $r>0$ such that
there is an isometric embedding
\[ \Theta:(\R \times \Sigma_r,-\td t^2 +h) \lra (M,g_{\Grad \ol{\bb}{}^{\eta}}) \]
as in the proof of Proposition~\ref{pr:Lsplit},
where $\Sigma_r \subset Z$ is the open ball with center $\eta(0)$ and radius $r$
with respect to $g_{\Grad \ol{\bb}{}^{\eta}}|_{TZ}$
and $h$ is the restriction of $g_{\Grad \ol{\bb}{}^{\eta}}|_{TZ}$ to $\Sigma_r$.
Proposition~\ref{pr:Lsplit} ensures $r_0>0$.
Note that, for $x \in \Sigma_{r_0}$, $\zeta_x=\Theta(\cdot,x)$ is
a timelike straight line of unit speed and $I(\zeta_x)=I(\eta) \supset \Theta(\R \times \Sigma_{r_0})$.
Moreover, $\bb^{\zeta_x}=\bb^{\eta}=-\ol{\bb}{}^{\eta}$ holds on $\Theta(\R \times \Sigma_{r_0})$
since it follows from Lemma~\ref{lm:GH2.5}\eqref{zeta1} that $\bb^{\zeta_x} \ge \bb^{\eta}$
and similarly $\bb^{\eta} \ge \bb^{\zeta_x}$.

If $\Theta(\R \times \Sigma_{r_0}) \neq M$, then $r_0<\infty$.
For $x \in \del \Sigma_{r_0}$, take a sequence $(x_k)_{k \in \N}$ in $\Sigma_{r_0}$ converging to $x$.
Then the straight line $\zeta_{x_k}$ converges to a straight line $\zeta_x$
passing through $x$ (thanks to Lemma~\ref{lm:GH2.3}),
and we have $I(\zeta_x)=I(\eta)$ and $\bb^{\zeta_x}=\bb^{\eta}$ as above.
Applying Proposition~\ref{pr:Lsplit} to $\zeta_x$,
one can extend $\Theta$ to $\R \times (\Sigma_{r_0} \cup B^Z_{x,\tau})$ for some $\tau>0$,
where $B^Z_{x,\tau} \subset Z$ is the open ball with center $x$ and radius $\tau$
with respect to $g_{\Grad\ol{\bb}{}^{\eta}}|_{TZ}$.
Note that, for $z \in B^Z_{x,\tau} \setminus \Sigma_{r_0}$,
$\xi_z=\Theta(\cdot,z)$ satisfies $I(\xi_z)=I(\zeta_x)=I(\eta)$
and $\bb^{\xi_z}=\bb^{\zeta_x}=\bb^{\eta}$.
We also observe that, for any $y \in \Sigma_{r_0}$,
$\zeta_y$ and $\xi_z$ do not intersect since $\bb^{\zeta_y}=\bb^{\xi_z}$.

Applying the above argument to each $x \in \del \Sigma_{r_0}$,
we can extend $\Theta$ to $\R \times \Sigma_r$ for some $r>r_0$,
a contradiction to the choice of $r_0$.
Therefore we conclude that $\Theta(\R \times \Sigma_r)=M$ for some $r \in (0,\infty]$
and $\Theta$ gives the desired isometry
(in particular, we have $\Sigma=Z$).

The completeness of $\Sigma$ follows from the timelike geodesic completeness of $(M,L)$.
The last assertion $(\Psi \circ \zeta)' \equiv 0$ has been shown in Proposition~\ref{pr:Lsplit}.
$\qedd$
\end{proof}

Similarly to Corollary~\ref{cr:Ltrans}, we obtain the following.

\begin{corollary}[Isometric translations]\label{cr:Gtrans}
Let $(M,L,\Psi)$ be as in Theorem~$\ref{th:Gsplit}$.
Then, in the product structure $M=\R \times \Sigma$, the translations
\[ \Phi_t(s,x):=(s+t,x), \qquad (s,x) \in \R \times \Sigma,\ t \in \R, \]
are isometric transformations of $(M,L)$.
\end{corollary}

We close the article with some remarks.

\begin{remark}[The case of $\psi_{\fm}$]\label{rm:msplit}
If we consider the weight funtion $\psi_{\fm}$ induced from a measure $\fm$ on $M$
(recall Remarks~\ref{rm:meas} and \ref{rm:psi_m}),
then we have $(\psi_{\fm} \circ \dot{\zeta})' \equiv 0$ which implies that
the push-forward of $\fm$ by $\Phi_t$ coincides with $\fm$.
Hence, $\Phi_t:(M,L,\fm) \lra (M,L,\fm)$ is measure-preserving.
\end{remark}

\begin{remark}[Non-Berwald case]\label{rm:nonBer}
In the non-Berwald case, one cannot obtain isometric translations as in Corollary~\ref{cr:Gtrans}.
This is because, since our hypotheses are imposed only in future-directed timelike directions,
one can modify $L$ in spacelike directions to violate the isometry of $\Phi_t$.
\end{remark}

\begin{remark}[The case of $N=0$]\label{rm:N=0}
In the extremal case of $N=0$,
it was established in \cite{WW2} that a Lorentzian manifold $(M,g,\Psi)$
has a \emph{warped product} splitting (see \cite{Wy} for the Riemannian case).
It is unclear if we can generalize it to the Lorentz--Finsler setting.
In our proof of Proposition~\ref{pr:max},
there is a difficulty in the construction of $\Sigma$ in Step~\ref{step4}
when $\R \times \widehat{\Sigma}$ is a warped product.
\end{remark}

{\it Acknowledgements}.
YL and SO were supported by JSPS Grant-in-Aid for Scientific Research (KAKENHI) 19H01786.
EM thanks the Department of Mathematics of Osaka University for kind hospitality.

{\small

}

\end{document}